\documentclass[pdflatex,sn-mathphys-num]{sn-jnl}

\usepackage{graphicx}%
\usepackage{amsmath,amssymb,amsfonts}%
\usepackage{amsthm}%
\usepackage{mathtools}%
\usepackage[title]{appendix}%
\usepackage{mathrsfs}%
\usepackage{xcolor}%
\usepackage{textcomp}%
\usepackage{manyfoot}%
\usepackage{booktabs}%

\newcommand{\argmin}{\mathop{\rm argmin}\limits}
\newcommand{\drv}[2]{\frac{\rmd#1}{\rmd#2}}
\newcommand{\calC}{{\mathcal C}}
\newcommand{\calE}{{\mathcal E}}
\newcommand{\bbN}{{\mathbb N}}
\newcommand{\bbR}{{\mathbb R}}
\newcommand{\rmd}{{\mathrm d}}

\theoremstyle{thmstyleone}%
\newtheorem{theorem}{Theorem}%
\newtheorem{proposition}{Proposition}%
\newtheorem{lemma}{Lemma}%

\theoremstyle{thmstyletwo}%
\newtheorem{example}{Example}%
\newtheorem{remark}{Remark}%

\theoremstyle{thmstylethree}%
\newtheorem{assumption}{Assumption}%

\begin{document}

\title[Convergence Analysis of an Upwind Finite-Difference Scheme for HJB Equations]{Convergence Analysis of an Upwind Finite-Difference Scheme for Hamilton--Jacobi--Bellman Equations}

\author*[1]{\fnm{Daisuke} \sur{Inoue}}\email{daisuke-inoue@mosk.tytlabs.co.jp}

\author[2]{\fnm{Nao} \sur{Hamamuki}}

\author[1]{\fnm{Yuji} \sur{Ito}}

\author[3]{\fnm{Takahito} \sur{Kashiwabara}}

\author[3]{\fnm{Norikazu} \sur{Saito}}

\author[1]{\fnm{Hiroaki} \sur{Yoshida}}

\affil*[1]{\orgname{Toyota Central R\&D Labs., Inc.}, \orgaddress{\street{41-1 Yokomichi}, \city{Nagakute}, \postcode{480-1192}, \state{Aichi}, \country{Japan}}}

\affil[2]{\orgdiv{Department of Mathematics, Faculty of Science}, \orgname{Hokkaido University}, \orgaddress{\street{Kita 10, Nishi 8, Kita-ku}, \city{Sapporo}, \postcode{060-0810}, \state{Hokkaido}, \country{Japan}}}

\affil[3]{\orgdiv{Graduate School of Mathematical Sciences}, \orgname{The University of Tokyo}, \orgaddress{\street{3-8-1 Komaba, Meguro-ku}, \city{Tokyo}, \postcode{153-8914}, \country{Japan}}}

\abstract{
  We study the convergence of an upwind finite-difference scheme for Hamilton--Jacobi--Bellman equations on a one-dimensional periodic domain, considering both the value function and the feedback control selected by the discrete pointwise minimization.
  For classical value functions whose time and spatial derivatives are H\"{o}lder continuous in their respective variables, we bound the error at the grid points by a constant times the sum of the time step and spatial mesh width each raised to their common H\"{o}lder exponent.
  This yields first-order convergence when these derivatives are Lipschitz continuous and the meshes are refined proportionally.
  Under additional spatial regularity of the problem data, we also prove convergence of both one-sided difference quotients of the discrete value function to the spatial derivative of the value function: the integral over space and time of the absolute error tends to zero for each difference quotient.
  We further establish convergence of the discrete controls to the continuous feedback control when the continuous pointwise minimizer is unique, regardless of how the discrete minimizers are selected.
  These theoretical results are supported by numerical experiments that illustrate the accuracy of the value-function approximation and the behavior of the computed feedback controls.
}


\maketitle

\section{Introduction}

Hamilton--Jacobi--Bellman (HJB) equations characterize optimal control problems through nonlinear partial differential equations whose Hamiltonians are defined by pointwise minimization over the control variable~\cite{KirkOptimal2004,Fleming2006Controlled}.
Their numerical solution therefore combines the approximation of a nonlinear PDE with repeated minimization over the control variable.
A wide range of numerical methods have been developed, including finite-difference schemes~\cite{Osher1991HighOrder,Briggs2002FiniteDifference,Fleming2006Controlled,Shu2007High,Bokanowski2010Convergence}, finite-element and discontinuous-Galerkin methods~\cite{Milner1996Mixed,Hu1999Discontinuous,Cheng2007discontinuous,Li2010Central,Jensen2013Convergencea,JaroszkowskiJensen2022FiniteElement}, and semi-Lagrangian methods~\cite{FalconeGiorgi1999Approximation,Falcone2002SemiLagrangian,Bokanowski2013Adaptive,Falcone2013SemiLagrangian}.
In this paper, we focus on an upwind finite-difference approach.
In one spatial dimension, the upwind scheme uses a three-point spatial stencil and pointwise minimization to compute the control.
Under a suitable condition on the mesh sizes, the scheme is monotone and hence compatible with the viscosity-solution framework for HJB equations~\cite{Wang2000upwind,Sun2015Convergence}.

In control applications, both the accuracy of the numerical value function and the convergence of the computed feedback control need to be assessed.
First, an error estimate for the numerical value function tells us how accurately the optimal cost has been computed.
It also helps determine the mesh size needed for a desired accuracy.
Second, implementing feedback control requires a control law obtained by pointwise minimization using the spatial gradient of the value function.
Uniform convergence of the value function does not in general imply convergence of the difference quotients used to approximate this gradient, so convergence of the computed control requires a separate analysis.
We therefore study error estimates for the value function and convergence of the numerical difference quotients and feedback controls for an upwind finite-difference scheme on a one-dimensional periodic domain.

Our first contribution concerns the convergence rate of the value-function approximation for the upwind scheme.
For monotone approximations, the viscosity-solution framework establishes qualitative convergence~\cite{Barles1991Convergence}, while a half-order error estimate under Lipschitz regularity is classical~\cite{Crandall1984Two}.
Sharper results under additional regularity or structure include first-order estimates under semiconcavity and regularity-dependent estimates for finite-horizon deterministic control problems~\cite{Falcone1987Numerical,FalconeGiorgi1999Approximation}.
Regularity-dependent rates and quantitative bounds for broad classes of monotone HJB schemes are available~\cite{DupuisJames1998Rates,BarlesJakobsen2002ConvergenceRate,BarlesJakobsen2007ErrorBounds,Krylov2005Rate}.
For the particular upwind discretization considered here, stability and first-order truncation error for smooth solutions have been studied~\cite{Wang2000upwind}, while qualitative convergence of the numerical value function to the viscosity solution is known without a quantitative rate~\cite{Sun2015Convergence}.
For the same discretization, we consider classical solutions with $\partial_t v$ H\"{o}lder continuous in time and $\partial_x v$ H\"{o}lder continuous in space for a common exponent $\eta\in(0,1]$.
We prove that the grid-point error is of order $O(\Delta t^\eta+\Delta x^\eta)$.
Under proportional refinement of the temporal and spatial meshes, this becomes $O(\Delta t^\eta)$ and in particular gives first-order convergence when the first derivatives are Lipschitz continuous.

Our second contribution concerns the convergence of the numerical difference quotients and the feedback controls computed from them.
For Markov-chain approximations of structured problems with control-affine dynamics and a quadratic control cost, feedback convergence is known on regions where the feedback is well defined and smooth~\cite{DupuisSzpiro2001Feedback,BoueDupuis1999Markov}.
A second-order result for the solution, its gradient, and the associated feedback on smooth subsets is also available for a class of stationary first-order Hamilton--Jacobi equations~\cite{SzpiroDupuis2002SecondOrder}.
Other finite-difference analyses establish derivative information in different settings, including interior first-order-difference estimates for uniformly elliptic Bellman equations~\cite{Krylov2012Interior} and gradient-convergence results for Lax--Friedrichs approximations of convex superlinear Hamiltonians~\cite{Soga2020Stochastic}.
Semiconcavity-based analyses provide estimates for numerical derivatives and approximate solutions~\cite{LinTadmor2001L1Stability,GosseJames2002Convergence,AchdouCamilliCorrias2014Numerical,BastiCamilli2025LpEstimates}.
In the upwind control discretization considered here, the numerical Hamiltonian is generally nonsmooth and the discrete minimizer may be nonunique even when the continuous minimizer is unique.
Under spatial regularity assumptions on the data, we prove strong $L^1$ convergence of both one-sided numerical difference quotients to the spatial derivative of the value function.
This result does not require the value function to be a classical solution.
If the continuous pointwise minimizer is unique, this further implies almost-everywhere convergence of every selection of discrete minimizers to the continuous feedback control along a subsequence.
Discrete semiconcavity and periodicity provide spatial total-variation bounds, which combine with qualitative value convergence through a spatial averaging argument to give strong convergence of the difference quotients.

The remainder of this paper is organized as follows.
Section~\ref{sec:review} introduces the HJB problem and the upwind finite-difference scheme.
The convergence results are stated in Section~\ref{sec:convergence} and proved in Section~\ref{sec:proof}.
Section~\ref{sec:numeric} presents a smooth periodic example and a sparse-control example with nonunique minimizers.
Section~\ref{sec:conclusion} concludes the paper.

\textbf{Notation.}
We write $\nabla_x$ for the spatial gradient, $\partial_x$ for the one-dimensional spatial derivative, and $\partial_t$ for the temporal derivative.
$\calC^k(D)$ denotes the space of real-valued functions with continuous derivatives up to order $k$ on $D$.
We set $\calC(D)\coloneqq\calC^0(D)$ and write $\calC^\infty(D)$ for the space of smooth functions.
A hat distinguishes a discrete operator from its continuous counterpart.
For the periodic domain $\Omega$ introduced in Section~\ref{sec:scheme}, we identify functions in $\calC^k(\Omega)$ with their periodic lifts to $\bbR$.
For $1\le p\le\infty$, $L^p(D)$ denotes the usual Lebesgue space, and integrals and $L^p$ norms over $\Omega$ are taken over one period.
Unless stated otherwise, almost-everywhere statements are understood with respect to Lebesgue measure.

\section{Problem Formulation and Numerical Scheme}\label{sec:review}

This section reviews the HJB formulation of optimal control problems and introduces the upwind finite-difference scheme used in the convergence analysis.

\subsection{HJB Equation}\label{sec:hjb}

We define the optimal control problem and formally derive the HJB equation.
A derivation is given in \cite{Fleming2006Controlled}.
First, we consider the following dynamical system:
\begin{align}\label{eq:ODE_system}
  \begin{dcases}
    \drv{X}{s}(s)=f(X(s), a(s)), \qquad s \in[t, T], \\
    X(t) = x,
  \end{dcases}
\end{align}
Here $s\in[t,T]$ denotes time, with $t\ge0$ and terminal time $T>0$.
The state is $X(s)\in\bbR^n$.
The control $a(s)$ takes values in a nonempty bounded closed set $E\subseteq\bbR^m$.
The function $f:\bbR^n\times E\to\bbR^n$ denotes the system dynamics.
The set of admissible open-loop controls is
\begin{align}
  \calE
  \coloneqq
  \left\{a:[0,T]\to E:\ a\text{ is Lebesgue measurable}\right\}.
\end{align}
Since $E$ is bounded, every $a\in\calE$ is essentially bounded.
Next, we consider the following cost function for the system \eqref{eq:ODE_system}:
\begin{align}\label{eq:cost_func}
  J(a;x,t) = \int_{t}^{T} g(X(s), a(s)) \mathrm{d} s + v_T (X(T)),
\end{align}
where the function $g:\bbR^n\times E\to\bbR$ denotes the running cost, and $v_T:\bbR^n\to\bbR$ denotes the terminal cost.
Here, $X$ in \eqref{eq:cost_func} denotes the state trajectory following the dynamics \eqref{eq:ODE_system} under the control input $a$.
The optimal control problem seeks minimizers $a^{\mathrm{opt}}\in\calE$ of the cost function $J(a;x,0)$:
\begin{align}
  J(a^{\mathrm{opt}};x,0) = \inf_{a\in\calE} J(a;x,0).
\end{align}

The solution to the optimal control problem is characterized by the HJB equation.
We first define the \emph{value function}:
\begin{align}\label{eq:value_function}
  v(x,t) \coloneqq\inf_{a\in\calE} J(a;x,t).
\end{align}
The value function satisfies the following recursive property: for any $r\in[t,T]$,
\begin{align}\label{eq:dynamic_principle}
  v(x,t) = \inf_{a\in \calE}\left[\int_{t}^{r} g(X(s),a(s))\rmd s + v\left(X(r),r\right) \right],
\end{align}
where $X$ follows the dynamics \eqref{eq:ODE_system} under the control input $a$.
Equation~\eqref{eq:dynamic_principle} is called the \emph{dynamic programming principle}.
Define the Hamiltonian associated with the control problem by
\begin{align}\label{eq:continuous-Hamiltonian}
  H(x,p)
  \coloneqq
  \inf_{a\in E}\left\{f(x,a)^\top p+g(x,a)\right\}.
\end{align}
Setting $r=t+h$ for $h>0$, dividing both sides by $h$, and taking the limit as $h\to0$, we can derive the HJB equation~\citep{Fleming2006Controlled}:
\begin{align}
  \begin{split}\label{eq:HJB}
    \partial_t v(x,t) & =-H(x,\nabla_x v(x,t)),
    \quad (x,t)\in\bbR^n\times(0,T),            \\
    v(x,T)            & =v_T(x),
    \quad x\in\bbR^n.
  \end{split}
\end{align}
Thus, \eqref{eq:HJB} is a first-order nonlinear partial differential equation with a terminal condition.

The HJB equation does not necessarily have a classical solution.
We therefore use the \emph{viscosity solution} concept.
The function $v$ is said to be a viscosity solution if $v$ is continuous on $\bbR^n \times [0,T]$, satisfies $v(x,T)=v_T(x)$ for every $x\in\bbR^n$, and satisfies both of the following:
\begin{enumerate}
  \item Viscosity subsolution: For any $\phi\in\calC^1(\bbR^n\times(0,T))$,
        $$
          -\partial_t \phi(x_0,t_0) \le H(x_0,\nabla_x \phi(x_0,t_0)),
        $$
        at every $(x_0,t_0)\in\bbR^n\times(0,T)$ which is a local maximizer of $v-\phi$.
  \item Viscosity supersolution: For any $\phi\in\calC^1(\bbR^n\times(0,T))$,
        $$
          -\partial_t \phi(x_0,t_0) \ge H(x_0,\nabla_x \phi(x_0,t_0)),
        $$
        at every $(x_0,t_0)\in\bbR^n\times(0,T)$ which is a local minimizer of $v-\phi$.
\end{enumerate}

\begin{remark}\label{rem:optimal_control_existence}
  The pointwise minimization defining the Hamiltonian connects the HJB equation with optimal controls.
  If the value function is a classical solution and an optimal open-loop control $a^{\mathrm{opt}}$ exists, then this control and its state trajectory $X^{\mathrm{opt}}$ satisfy
  \begin{align}
    a^{\mathrm{opt}}(s)
    \in
    \argmin_{a\in E}
    \left\{
    f(X^{\mathrm{opt}}(s),a)^\top
    \nabla_x v(X^{\mathrm{opt}}(s),s)
    +g(X^{\mathrm{opt}}(s),a)
    \right\}
  \end{align}
  for a.e.\ $s\in[0,T]$; see \cite[Chapter~I, Section~6]{Fleming2006Controlled}.
  When the pointwise minimizer is unique for every $(x,p)\in\bbR^n\times\bbR^n$, we denote it by $a^*(x,p)$, so that
  \begin{align}
    \argmin_{a\in E}
    \left\{
    f(x,a)^\top p+g(x,a)
    \right\}
    =\{a^*(x,p)\}.
  \end{align}
  In this case, the pointwise minimum condition becomes
  \begin{align}
    a^{\mathrm{opt}}(s)
    =
    a^*\left(
    X^{\mathrm{opt}}(s),
    \nabla_x v(X^{\mathrm{opt}}(s),s)
    \right)
    \quad\text{for a.e.\ }s\in[0,T].
  \end{align}
  Conversely, if the feedback law
  \begin{align}
    (x,t)\longmapsto a^*(x,\nabla_x v(x,t))
  \end{align}
  yields an admissible closed-loop control satisfying the hypotheses of the verification theorem, then the resulting control is optimal; see \cite[Chapter~I, Section~5]{Fleming2006Controlled}.
\end{remark}

\subsection{Upwind Finite-Difference Scheme}\label{sec:scheme}

Section~\ref{sec:hjb} reviewed the general optimal-control and HJB formulation on $\bbR^n$.
We now specialize to one spatial dimension with period $2$ and introduce the corresponding upwind finite-difference scheme.
From this point onward, $\partial_x$ denotes the one-dimensional spatial derivative.
We set $\Omega\coloneqq\mathbb{R}/(2\mathbb{Z})$, the one-dimensional torus of period $2$, and equip it with the distance
$d_\Omega(x,y)\coloneqq\min_{q\in\mathbb{Z}}|x-y+2q|$.
Functions on $\Omega$ are identified with their $2$-periodic lifts to $\bbR$.
Accordingly, membership in $\calC^k(\Omega)$ is understood through the corresponding $2$-periodic lift to $\bbR$.
In particular, for every $(x,a)\in\bbR\times E$, the lifted data satisfy
\begin{align}
  f(x+2,a)=f(x,a),\qquad g(x+2,a)=g(x,a),\qquad v_T(x+2)=v_T(x).
\end{align}

\begin{assumption}[Basic periodic HJB conditions]\label{assmp:HJB}
  \begin{enumerate}
    \item The control set $E\subset\bbR^m$ is nonempty and compact.
    \item The functions
          \begin{align}
            f:\Omega\times E\longrightarrow\bbR,
            \qquad
            g:\Omega\times E\longrightarrow\bbR,
          \end{align}
          are continuous.
    \item There exists $L_f>0$ such that
          \begin{align}
            |f(x,a)-f(y,a)|\le L_f d_\Omega(x,y),
          \end{align}
          for every $x,y\in\Omega$ and $a\in E$.
    \item The terminal cost satisfies $v_T\in\calC(\Omega)$.
  \end{enumerate}
\end{assumption}

\begin{proposition}\label{prop:periodic_HJB}
  Under Assumption~\ref{assmp:HJB}, the value function $v$, defined using the $2$-periodic lifts of $f$, $g$, and $v_T$, is bounded and uniformly continuous and is the unique bounded uniformly continuous viscosity solution of the HJB equation \eqref{eq:HJB} on $\bbR\times[0,T]$.
  Moreover,
  \begin{align}\label{eq:value-periodicity}
    v(x+2,t)=v(x,t),
  \end{align}
  for every $(x,t)\in\bbR\times[0,T]$.
  Hence, $v$ may equivalently be regarded as the unique continuous periodic viscosity solution on $\Omega\times[0,T]$.
\end{proposition}

\begin{proof}
  We apply Theorem~\ref{thm:FS-wellposed} of Appendix~\ref{app:standard}, the specialization of \cite[Chapter~II, Theorem~10.1]{Fleming2006Controlled} recorded there, to the $2$-periodic lifts of the data.
  We first verify its boundedness, uniform-continuity, and Lipschitz assumptions.
  Since $\Omega\times E$ is compact, Assumption~\ref{assmp:HJB} implies that $f$ and $g$ are bounded.
  The periodic lift of $g$ is bounded and uniformly continuous, and the same is true of the periodic lift of $v_T$ because $v_T$ is continuous on the compact set $\Omega$.
  Moreover, for the periodic lift of $f$ and every $x,y\in\bbR$ and $a\in E$,
  \begin{align}
    |f(x,a)-f(y,a)|
    \le L_f d_\Omega(x,y)
    \le L_f|x-y|.
  \end{align}
  Thus, the dynamics are Lipschitz continuous in $x$, uniformly in $a$.
  Theorem~\ref{thm:FS-wellposed} therefore shows that the value function is the unique bounded uniformly continuous viscosity solution of \eqref{eq:HJB} on $\bbR\times[0,T]$.

  It remains to prove the periodicity in \eqref{eq:value-periodicity}.
  Fix an admissible control $a\in\calE$, and let $X(s;x,t,a)$ denote the associated state trajectory starting from $x$ at time $t$.
  Since $f$ is $2$-periodic, the shifted path $X(s;x,t,a)+2$ satisfies
  \begin{align}
    \frac{\rmd}{\rmd s}\bigl[X(s;x,t,a)+2\bigr]
    =f\bigl(X(s;x,t,a)+2,a(s)\bigr),
  \end{align}
  and has initial value $x+2$ at time $t$.
  By uniqueness of solutions of the state equation,
  \begin{align}
    X(s;x+2,t,a)=X(s;x,t,a)+2.
  \end{align}
  The $2$-periodicity of $g$ and $v_T$ therefore gives
  \begin{align}
    \begin{split}
      J(a;x+2,t)
       & =\int_t^T g\bigl(X(s;x+2,t,a),a(s)\bigr)\,\rmd s
      +v_T\bigl(X(T;x+2,t,a)\bigr)                        \\
       & =\int_t^T g\bigl(X(s;x,t,a)+2,a(s)\bigr)\,\rmd s
      +v_T\bigl(X(T;x,t,a)+2\bigr)                        \\
       & =J(a;x,t).
    \end{split}
  \end{align}
  Taking the infimum over $a\in\calE$ proves \eqref{eq:value-periodicity}.
  Finally, every continuous periodic function on $\bbR\times[0,T]$ is bounded and uniformly continuous, so uniqueness in the bounded uniformly continuous class also gives uniqueness in the continuous periodic class.
\end{proof}

Periodicity will yield the bound on the spatial total variation of the numerical slopes used in the direct gradient-convergence argument underlying feedback control convergence.
For the discretization, we represent $\Omega$ by the half-open interval $[-1,1)$ and use a periodic grid with $N_x$ distinct spatial points:
\begin{align}
  x_i=-1+i\Delta x\quad (i=0,\ldots,N_x-1),
  \qquad
  t_j=j\Delta t\quad (j=0,\ldots,N_t),
\end{align}
where $\Delta x\coloneqq 2/N_x$ and $\Delta t\coloneqq T/N_t$ for positive integers $N_x$ and $N_t$.
All spatial indices are understood modulo $N_x$.

Next, we discretize the HJB equation.
The approximate value of $v(x_i,t_j)$ is denoted by $V_{i,j}$, and a minimizer of the corresponding discrete pointwise minimization is denoted by $A^*_{i,j}$.
For a scalar $r$, set $r^+\coloneqq\max\{r,0\}$ and $r^-\coloneqq\min\{r,0\}$, and define $f^\pm(x,a)\coloneqq(f(x,a))^\pm$.
Define $D^+V_{i,j}\coloneqq (V_{i+1,j}-V_{i,j})/\Delta x$ and $D^-V_{i,j}\coloneqq (V_{i,j}-V_{i-1,j})/\Delta x$.
Then, the upwind finite-difference scheme for the HJB equation \eqref{eq:HJB} is defined as
\begin{align}
   & \frac{V_{i,j-1}-V_{i,j}}{\Delta t} = f^+(x_i,A_{i,j}^{*}) D^+V_{i,j}+f^{-}(x_i,A_{i,j}^{*}) D^-V_{i,j}+ g(x_i,A_{i,j}^{*}),\label{eq:upwind_HJB} \\
   & A_{i,j}^{*} \in \argmin_{a \in E}\left\{f^{+}(x_i,a) D^+V_{i,j}+f^{-}(x_i,a) D^-V_{i,j}+ g(x_i,a)\right\},\label{eq:upwind_HJB_input}            \\
   & V_{i,N_t} = v_T (x_i).\label{eq:upwind_HJB_initial_value}
\end{align}
The discrete minimizers $A^*_{i,j}$ in \eqref{eq:upwind_HJB_input} are defined for $j=0,\ldots,N_t$, whereas the time update \eqref{eq:upwind_HJB} is performed for $j=1,\ldots,N_t$.
Thus, $A^*_{i,N_t}$ is defined from the terminal value $V_{i,N_t}$ and used in the first backward step, while $A^*_{i,0}$ is defined from $V_{i,0}$ but is not used in a time update.
For $x\in\Omega$ and $p^-,p^+\in\bbR$, define the numerical Hamiltonian by
\begin{align}\label{eq:numerical-Hamiltonian}
  \widehat H(x,p^-,p^+)
  \coloneqq
  \min_{a\in E}
  \left\{f^-(x,a)p^-+f^+(x,a)p^++g(x,a)\right\}.
\end{align}
Under Assumption~\ref{assmp:HJB}, the minima in \eqref{eq:continuous-Hamiltonian} and \eqref{eq:numerical-Hamiltonian} are attained, and
\begin{align}\label{eq:Hamiltonian-diagonal}
  \widehat H(x,p,p)=H(x,p).
\end{align}
The values with indices $i+1$ and $i-1$ in these formulas are evaluated modulo $N_x$; no duplicated endpoint is included in the grid.
Equation~\eqref{eq:upwind_HJB} can equivalently be written as
\begin{align}
  \begin{aligned}
    V_{i,j-1} & =\left(1- \alpha\left|f(x_i,A_{i,j}^{*})\right|\right) V_{i,j}                                                                                        \\
              & \quad + \alpha f^{+}(x_i,A_{i,j}^{*}) V_{i+1,j} - \alpha f^{-}(x_i,A_{i,j}^{*}) V_{i-1,j} +\Delta t g(x_i,A_{i,j}^{*}),\label{eq:upwind_HJB_explicit}
  \end{aligned}
\end{align}
where we have defined
\begin{align}
  \alpha\coloneqq\frac{\Delta t}{\Delta x}.
\end{align}
This is an explicit scheme, where the solution is obtained by alternately solving the minimization step \eqref{eq:upwind_HJB_input} and the time evolution step \eqref{eq:upwind_HJB}.

The analysis developed below addresses the two issues discussed in the Introduction for the upwind scheme.
We first quantify the approximation of the value function when additional classical regularity is available.
We then investigate the passage from value-function convergence to convergence of the numerical difference quotients and of the controls selected by the discrete minimization.
The assumptions and results are stated in Section~\ref{sec:convergence}.

\begin{remark}
  Under a Courant--Friedrichs--Lewy (CFL) condition on the mesh ratio, the scheme \eqref{eq:upwind_HJB}--\eqref{eq:upwind_HJB_input} is monotone.
  The upwind scheme introduced in \cite{Crandall1984Two} differs in definition from the one considered here.
  Specifically, their scheme requires a priori knowledge of a point at which the derivative of the Hamiltonian changes sign, or a priori bounds on $\partial_x v$ for the CFL condition.
  In contrast, our scheme is constructed directly from the functions $f$ and $g$, and its monotonicity can be checked without such a priori information.
\end{remark}

\begin{remark}
  Under Assumption \ref{assmp:HJB}, both the pointwise minimization problem in the HJB equation \eqref{eq:HJB} and the discrete minimization problem in the upwind scheme \eqref{eq:upwind_HJB_input} admit minimizers.
  For each $(x,p)\in\Omega\times\bbR$, the map $a\mapsto f(x,a)p+g(x,a)$ is continuous on the compact set $E$, and hence attains its minimum by the Weierstrass theorem.
  Therefore, the corresponding argmin set is nonempty.
  Similarly, for given $(x,p^-,p^+)$, the discrete objective function
  $f^-(x,a)p^- + f^+(x,a)p^+ + g(x,a)$
  is continuous in $a$ on $E$, since $f^\pm$ are obtained from $f$ by
  composition with continuous functions.
  Consequently, the minimizer $A_{i,j}^*$ also exists.
\end{remark}

\begin{remark}\label{rem:discrete_nonunique}
  Even if the pointwise minimizer in the HJB equation is unique, the minimizer of the discrete upwind problem \eqref{eq:upwind_HJB_input}
  need not be unique.
  For example, let $E=[-1,1]$, $f(x,a)=a$, and $g(x,a)=a^2$.
  Then the objective function $a\mapsto ap+a^2$ is strictly convex, and hence admits a
  unique minimizer for every $p\in\bbR$.
  However, the discrete objective $a\mapsto a^+p^+ + a^-p^-+a^2$ can have multiple minimizers.
  Indeed, if $p^+=-2$ and $p^-=2$, then this objective is minimized both at $a=1$ and $a=-1$.
\end{remark}

\section{Convergence Results}\label{sec:convergence}

In this section, we present convergence results for the one-dimensional finite-difference scheme in \eqref{eq:upwind_HJB}--\eqref{eq:upwind_HJB_initial_value} introduced in Section~\ref{sec:scheme}.
Proofs are given in Section~\ref{sec:proof} and, for Proposition~\ref{prop:value_uniform_convergence}, in Appendix~\ref{app:BS-scheme}.
Under Assumption~\ref{assmp:HJB}, define the finite constant
\begin{align}\label{eq:F0-const}
  F_0\coloneqq\sup_{(x,a)\in\Omega\times E}|f(x,a)|.
\end{align}
First, we consider the case where the HJB equation has a classical solution and establish a regularity-dependent convergence rate.

\begin{theorem}\label{thm:convergence}
  Suppose that Assumption~\ref{assmp:HJB} holds and that the following CFL condition is satisfied:
  \begin{align}\label{eq:CFL}
    \alpha F_0 \le 1.
  \end{align}
  Furthermore, suppose that the value function $v(x,t)$ associated with \eqref{eq:HJB} is a classical solution on $\bbR\times(0,T)$ and that its first derivatives $\partial_t v$ and $\partial_x v$ extend continuously to $\bbR\times[0,T]$.
  For some common exponent $\eta\in(0,1]$, suppose in addition that the extended derivative $\partial_t v$ is H\"{o}lder continuous in $t\in[0,T]$, uniformly for $x$ in one spatial period.
  Suppose also that the extended derivative $\partial_x v$ is H\"{o}lder continuous in $x$ on one spatial period with the same exponent, uniformly in $t\in[0,T]$.
  Then, there exist $C_t$ and $C_x$, independent of the mesh sizes, such that for any $(\Delta x, \Delta t)$ satisfying \eqref{eq:CFL}, the following holds:
  \begin{align}
    \sup_{i,j} |V_{i,j} - v(x_i,t_j)| \le  C_t(\Delta t)^\eta + C_x(\Delta x)^\eta.
  \end{align}
\end{theorem}

The classical and H\"{o}lder regularity required in Theorem~\ref{thm:convergence} goes beyond that needed in the standard viscosity-solution framework.
When the ratio $\Delta t/\Delta x$ is kept constant during refinement, the estimate in Theorem~\ref{thm:convergence} is of order $O((\Delta t)^\eta)$; it therefore improves on the general half-order rate when $\eta>1/2$, and yields first-order convergence when $\eta=1$.
Section~\ref{sec:numeric} examines the first-order case numerically.

\begin{remark}\label{rem:value-convergence-literature}
  General viscosity-solution theory gives qualitative convergence of monotone, stable, and consistent schemes~\cite{Barles1991Convergence}, while \cite{Crandall1984Two} provides a half-order error estimate under Lipschitz regularity.
  For the upwind scheme considered here, \cite{Wang2000upwind} discusses stability and first-order truncation error for smooth solutions, while related convergence analysis appears in \cite{Sun2015Convergence}.
  Theorem~\ref{thm:convergence} complements these results by making the rate explicitly dependent on the H\"{o}lder regularity of the first derivatives.
\end{remark}

We now turn to the convergence of the numerical difference quotients and, subsequently, of the controls.
For this purpose, we consider a sequence of meshes indexed by $k\in\bbN$.
Let $N_{x,k}$ and $N_{t,k}$ be positive integers satisfying
\begin{align}
  N_{x,k}\longrightarrow\infty,
  \qquad
  N_{t,k}\longrightarrow\infty,
\end{align}
and define
\begin{align}\label{eq:mesh_sequence}
  \Delta x_k\coloneqq\frac{2}{N_{x,k}},\qquad
  \Delta t_k\coloneqq\frac{T}{N_{t,k}},\qquad
  \alpha_k\coloneqq\frac{\Delta t_k}{\Delta x_k}.
\end{align}
Thus, $\Delta x_k\to0$ and $\Delta t_k\to0$ as $k\to\infty$; no fixed ratio between them is imposed here.
For each fixed $k$, we suppress the mesh index in fixed-mesh arguments by writing
\begin{align}
  N_x=N_{x,k},\qquad
  N_t=N_{t,k},\qquad
  \Delta x=\Delta x_k,\qquad
  \Delta t=\Delta t_k,\qquad
  \alpha=\alpha_k.
\end{align}
We also suppress the mesh index on the grid points when no confusion can arise, while retaining the $k$-indexed notation in statements concerning the mesh sequence.
We use the periodic lift $x_{i+qN_x}=x_i+2q$ for $i=0,\ldots,N_x-1$ and $q\in\mathbb{Z}$, and set $x_{i\pm1/2}=x_i\pm\Delta x/2$.
Define the periodic grid projection
\begin{align}
  \pi_k:\Omega\longrightarrow\{x_0,\ldots,x_{N_x-1}\}
\end{align}
by $\pi_k(x)=x_i$ when $x$ belongs to the image of the half-open interval $[x_{i-1/2},x_{i+1/2})$ under the quotient map $\mathbb{R}\to\Omega$.
These half-open intervals form a disjoint partition after passage to the quotient, including the interval that crosses the endpoints of $[-1,1)$, and
\begin{align}\label{eq:grid-projection-distance}
  d_\Omega(\pi_k(x),x)\le\frac{\Delta x_k}{2}.
\end{align}
We define a piecewise constant extension $A^*_k (x,t)$, $(x,t)\in \Omega\times [0,T]$, of the discrete control $A^*_{i,j}$ by
\begin{align}
  A^*_k(x,t) \coloneqq A^*_{i,j}
  \quad \text{when }\pi_k(x)=x_i\text{ and }t\in[t_j,t_{j+1}).\label{eq:interpolated_input}
\end{align}
The same procedure is applied to the value function $V$ and its spatial difference quotients $D^\pm V$:
\begin{align}
  V_k (x,t)      & \coloneqq V_{i,j}
                 &                         & \text{when }\pi_k(x)=x_i\text{ and }t\in[t_j,t_{j+1}),                               \\
  D^\pm V_k(x,t) & \coloneqq D^\pm V_{i,j}
                 &                         & \text{when }\pi_k(x)=x_i\text{ and }t\in[t_j,t_{j+1}).\label{eq:interpolated_slopes}
\end{align}
In these three definitions, $i=0,\ldots,N_x-1$ and the temporal interval index is $j=0,\ldots,N_t-1$.
At the terminal time, we complete the definitions by setting, whenever $\pi_k(x)=x_i$,
\begin{align}
  A_k^*(x,T)\coloneqq A_{i,N_t}^*,\qquad
  V_k(x,T)\coloneqq V_{i,N_t},\qquad
  D^\pm V_k(x,T)\coloneqq D^\pm V_{i,N_t}.
\end{align}

The argument for convergence of the numerical difference quotients combines spatial total-variation bounds with convergence of the numerical value functions to $v$.
Discrete spatial averaging then gives strong convergence of these difference quotients.
Since the analysis below does not assume the classical regularity required in Theorem~\ref{thm:convergence}, we use instead the standard qualitative convergence in the viscosity-solution framework recalled in Remark~\ref{rem:value-convergence-literature}.
The following proposition is an application of the monotone-scheme convergence theorem of Barles and Souganidis~\cite[Theorem~2.1]{Barles1991Convergence}.
We verify its assumptions for the present periodic scheme, including the boundary formulation used in that theorem.

\begin{proposition}[Qualitative convergence of the value approximation]\label{prop:value_uniform_convergence}
  Suppose that Assumption~\ref{assmp:HJB} holds.
  Consider the periodic mesh sequence defined in \eqref{eq:mesh_sequence}, and suppose that, for every sufficiently large $k$,
  \begin{align}\label{eq:value-convergence-CFL}
    \alpha_kF_0\le1.
  \end{align}
  Then,
  \begin{align}\label{eq:value-uniform-convergence}
    V_k\longrightarrow v
    \qquad\text{uniformly on }\Omega\times[0,T].
  \end{align}
\end{proposition}

The proof of Proposition~\ref{prop:value_uniform_convergence} is given in Appendix~\ref{app:BS-scheme}.

We now impose the additional spatial regularity needed for convergence of the numerical difference quotients.

\begin{assumption}[Uniform spatial regularity]\label{assmp:deriv}
  \begin{enumerate}
    \item For every $a\in E$, the functions $f(\cdot,a)$ and $g(\cdot,a)$ are of class $\calC^2(\Omega)$.
    \item The functions
          \begin{align}
            f,\quad \partial_x f,\quad \partial_{xx} f,\quad
            g,\quad \partial_x g,\quad \partial_{xx} g
          \end{align}
          are continuous on $\Omega\times E$.
    \item The terminal cost satisfies $v_T\in\calC^2(\Omega)$.
  \end{enumerate}
\end{assumption}
The second condition expresses that the spatial $\calC^2$ regularity is uniform with respect to $a\in E$.
Since $\Omega\times E$ is compact, the corresponding sup norms are finite.

Under Assumption~\ref{assmp:deriv}, define the finite constant
\begin{align}\label{eq:F1-const}
  F_1\coloneqq\sup_{(x,a)\in\Omega\times E}|\partial_x f(x,a)|.
\end{align}

\begin{proposition}\label{prop:convergence_derivative}
  Suppose that Assumptions~\ref{assmp:HJB} and \ref{assmp:deriv} hold.
  Consider the periodic mesh sequence defined in \eqref{eq:mesh_sequence}, and suppose that, for every sufficiently large $k$,
  \begin{align}\label{eq:semiconcavity_CFL}
    \alpha_kF_0+\Delta t_kF_1\le1.
  \end{align}
  Then, $v$ is Lipschitz continuous in $x$, uniformly in $t$, and its weak spatial derivative $\partial_x v\in L^\infty(\Omega\times(0,T))$ satisfies
  \begin{align}
    \lim_{k\to\infty} \|D^+V_k-\partial_x v \|_{L^1(\Omega\times(0,T))} = 0, \\
    \lim_{k\to\infty} \|D^-V_k-\partial_x v \|_{L^1(\Omega\times(0,T))} = 0.
  \end{align}
\end{proposition}

Spatial Lipschitz continuity of the value function also follows from standard value-function estimates under weaker assumptions; see \cite[Chapter~II, Remark~10.1]{Fleming2006Controlled}.
Here we instead recover this regularity from the discrete scheme.
The numerical solutions satisfy a semiconcavity bound with a constant independent of the mesh.
In one dimension, periodicity then gives uniform bounds on the slopes and their spatial total variation.
A spatial averaging argument relates the numerical slopes to difference quotients over a fixed spatial scale.
Value convergence and the spatial Lipschitz regularity of the limit then give strong $L^1$ convergence to $\partial_x v$.

For convergence of the controls, we separately impose uniqueness of the pointwise minimizer.
\begin{assumption}[Uniqueness of the pointwise minimizer]\label{assmp:control_uniqueness}
  For every $(x,p)\in\Omega\times\bbR$, the pointwise minimization problem has a unique minimizer $a^*(x,p)$; equivalently,
  \begin{align}\label{eq:optimization_problem}
    \argmin_{a\in E}\left[f(x,a)p+g(x,a)\right]=\{a^*(x,p)\}.
  \end{align}
\end{assumption}
Existence of a minimizer already follows from the compactness of $E$ and continuity with respect to $a$ under Assumption~\ref{assmp:HJB}; Assumption~\ref{assmp:control_uniqueness} requires only uniqueness and does not assume any additional regularity of the map $(x,p)\mapsto a^*(x,p)$.

\begin{example}\label{ex:quadratic_cost}
  Consider scalar controls in a nonempty compact interval $E\subset\bbR$ and functions
  \begin{align}
    f(x,a) & = b(x) + \beta(x) a,  \\
    g(x,a) & = q(x) + \rho(x) a^2.
  \end{align}
  Suppose that $b$, $\beta$, $q$, and $\rho$ are $2$-periodic functions of class $\calC^2(\Omega)$, that $\rho(x)>0$ for every $x\in\Omega$, and that $v_T\in\calC^2(\Omega)$.
  Then, Assumption~\ref{assmp:deriv} holds.
  Moreover, for each $(x,p)\in\Omega\times\bbR$, the objective function
  \begin{align}
    a\longmapsto [b(x)+\beta(x)a]p+q(x)+\rho(x)a^2
  \end{align}
  is strictly convex in $a$.
  Its restriction to the compact interval $E$ therefore has a unique minimizer, so Assumption~\ref{assmp:control_uniqueness} also holds.
\end{example}

The next theorem establishes convergence of the discrete controls to $a^*(x,\partial_x v(x,t))$.

\begin{theorem}\label{thm:convenience_input}
  Suppose that Assumptions~\ref{assmp:HJB}--\ref{assmp:control_uniqueness} hold and that the periodic mesh sequence \eqref{eq:mesh_sequence} satisfies \eqref{eq:semiconcavity_CFL} for every sufficiently large $k$.
  Then, there exist a subsequence of mesh indices $\{k_\nu\}_{\nu\in\bbN}$, with $k_1<k_2<\cdots$, and a set $S\subset\Omega\times[0,T]$ of full measure, both independent of the choice of discrete minimizers, such that the following properties hold:
  $v(\cdot,t)$ is differentiable at $x$ for every $(x,t)\in S$;
  for every $(x,t)\in S$ and every selection of the discrete minimizers in \eqref{eq:upwind_HJB_input}, the associated functions $A^*_{k_\nu}$ defined in \eqref{eq:interpolated_input} satisfy
  \begin{align}\label{eq:convenience_input}
    \lim_{\nu\to\infty} A^*_{k_\nu}(x,t)
    =a^*(x,\partial_x v(x,t)).
  \end{align}
\end{theorem}

Convergence of the value functions $V_k$ alone does not imply convergence of the minimizers selected by the discrete minimization problems.
Proposition~\ref{prop:convergence_derivative} proves strong $L^1$ convergence of the numerical difference quotients, which is used to compare the continuous and discrete minimization objectives in the proof of Theorem~\ref{thm:convenience_input}.
The Lipschitz conclusion of Proposition~\ref{prop:convergence_derivative} ensures that $\partial_x v(x,t)$ exists for a.e. $(x,t)\in\Omega\times[0,T]$, so the limit in Theorem~\ref{thm:convenience_input} is well defined at almost every point.
The argument for the difference quotients remains one-dimensional because semiconcavity of the discrete solutions is combined with periodicity to bound the total variation of the numerical slopes.
Extending these results to higher dimensions requires suitable control of multidimensional numerical gradients and is left for future work.

\begin{remark}
  For the continuous Bolza problem, semiconcavity of the value function follows under regularity and semiconcavity assumptions on the dynamics and costs; see \cite[Theorem~7.4.11]{Cannarsa2004Semiconcave}.
  At the discrete level, Lin and Tadmor use semiconcavity in their $L^1$-stability and error analysis of approximate solutions and Godunov-type schemes~\cite{LinTadmor2001L1Stability}.
  For space-dependent Hamiltonians, Gosse and James derive a discrete semiconcavity estimate for Lax--Friedrichs-type approximations under smoothness and convexity assumptions and a reinforced CFL condition~\cite{GosseJames2002Convergence}; they also note that their argument does not directly cover upwind-type numerical Hamiltonians with nonsmooth min/max structures.
  Achdou, Camilli, and Corrias derive a corresponding estimate for space-dependent semi-Lagrangian approximations by fixing a minimizer at the central point and reusing it at neighboring points~\cite{AchdouCamilliCorrias2014Numerical}.
  Basti and Camilli establish semiconcavity preservation for a broad class of monotone finite-difference schemes on the torus~\cite{BastiCamilli2025LpEstimates}; their setting assumes that $H=H(p)$ is smooth and convex and that the numerical Hamiltonian is $\calC^2$ and convex.
  By contrast, the numerical Hamiltonian $\widehat H$ considered here is space dependent and generally nonsmooth because it uses the positive and negative parts $f^\pm$ and pointwise minimization over $a$.
\end{remark}

\section{Proofs of the Convergence Results}\label{sec:proof}

In this section, we present the proofs of the results stated above.
Section~\ref{sec:proof-value} proves the quantitative error estimate of Theorem~\ref{thm:convergence}.
Section~\ref{sec:proof-gradients} develops the discrete semiconcavity and spatial variation estimates, then uses spatial averaging to establish the strong convergence of the numerical difference quotients in Proposition~\ref{prop:convergence_derivative}.
Section~\ref{sec:proof-controls} combines convergence of the difference quotients and pointwise minimizers to prove the control convergence in Theorem~\ref{thm:convenience_input}.

\subsection{Convergence of the Value Function}\label{sec:proof-value}

For the quantitative error estimate, we compare the continuous and discrete minimization problems at their respective minimizers and use the CFL condition to obtain a one-step error bound.
We then insert the H\"{o}lder estimates for the local residuals and sum backward from the zero terminal error.

\begin{proof}[Proof of Theorem \ref{thm:convergence}]
  We first prepare the continuous HJB equation for evaluation on the periodic grid.
  Proposition~\ref{prop:periodic_HJB} shows that $v$ is $2$-periodic and may be regarded as a function on $\Omega$.
  By the assumed classical regularity, its first derivatives are also $2$-periodic.
  Hence, the assumed H\"{o}lder bounds on one spatial period also hold with respect to the periodic distance $d_\Omega$.
  The error estimate uses the grid times $t_j$ with $j=1,\ldots,N_t$, including $t_{N_t}=T$, whereas \eqref{eq:HJB} is imposed only on the open time interval.
  It therefore remains to verify the HJB identity at the terminal grid time.

  To verify this terminal-time identity, fix $x\in\Omega$.
  For $t\in(0,T)$, the classical solution satisfies \eqref{eq:HJB}, that is,
  \begin{align}\label{eq:HJB-classical}
    -\partial_t v(x,t)=H(x,\partial_x v(x,t)).
  \end{align}
  Under Assumption~\ref{assmp:HJB}, the infimum defining $H$ is attained.
  Moreover, \eqref{eq:continuous-Hamiltonian} and the definition \eqref{eq:F0-const} of $F_0$ give
  \begin{align}
    |H(x,p)-H(x,q)|\le F_0|p-q|,
  \end{align}
  so the right-hand side of \eqref{eq:HJB-classical} is continuous in $\partial_x v(x,t)$.
  Since $\partial_t v$ and $\partial_x v$ extend continuously to $t=T$ by the hypotheses of Theorem~\ref{thm:convergence}, letting $t\uparrow T$ in \eqref{eq:HJB-classical} shows that the same identity holds at $t=T$.
  Thus \eqref{eq:HJB-classical} is available at every grid time $t_j$ with $j=1,\ldots,N_t$.

  We now compare the continuous and discrete minimization problems.
  For each pair $(i,j)$, choose an arbitrary minimizer
  \begin{align}
    a_{i,j}^{\mathrm c}\in\argmin_{a\in E}\left\{f(x_i,a)\partial_x v(x_i,t_j)+g(x_i,a)\right\}.
  \end{align}
  The superscripts $\mathrm c$ and $\mathrm d$ distinguish quantities evaluated at the continuous minimizer $a_{i,j}^{\mathrm c}$ and at the discrete minimizer $A_{i,j}^*$, respectively.
  Setting $\epsilon_{i,j}\coloneqq V_{i,j} - v(x_i,t_j)$, we have by
  \eqref{eq:upwind_HJB}, \eqref{eq:upwind_HJB_input}, and \eqref{eq:HJB-classical}
  \begin{align}
    \begin{split}
       & \frac{\epsilon_{i,j-1} - \epsilon_{i,j}}{\Delta t}-f^{+}\left(x_i,a_{i,j}^{\mathrm c}\right) \frac{\epsilon_{i+1,j}-\epsilon_{i,j}}{\Delta x}-f^{-}\left(x_i,a_{i,j}^{\mathrm c}\right) \frac{\epsilon_{i,j}-\epsilon_{i-1,j}}{\Delta x} \\
       & \le \frac{v(x_i,t_{j}) - v(x_i,t_{j-1})}{\Delta t}
      + f^+(x_i,a_{i,j}^{\mathrm c})\frac{v(x_{i+1},t_{j}) - v(x_i,t_j)}{\Delta x}                                                                                                                                                                \\
       & \quad
      + f^-(x_i,a_{i,j}^{\mathrm c})\frac{v(x_{i},t_{j}) - v(x_{i-1},t_j)}{\Delta x}
      + g(x_i,a_{i,j}^{\mathrm c})                                                                                                                                                                                                                \\
       & \quad + \left(
      - \partial_t v(x_i,t_j)
      - f(x_i,a_{i,j}^{\mathrm c})\partial_x v(x_i,t_j)
      -  g\left(x_i,a_{i,j}^{\mathrm c}\right)\right)                                                                                                                                                                                             \\
       & = r_{i,j}^{1} + r_{i,j}^{2,\mathrm c},
    \end{split}\label{eq:eps_1}
  \end{align}
  where we have defined
  \begin{align}
     & r_{i,j}^{1}\coloneqq
    \frac{v(x_i,t_{j}) - v(x_i,t_{j-1})}{\Delta t} -
    \partial_t v(x_i,t_j),  \\
    \begin{split}
       & r_{i,j}^{2,\mathrm c}\coloneqq f^+(x_i,a_{i,j}^{\mathrm c})\frac{v(x_{i+1},t_{j}) - v(x_i,t_j)}{\Delta x}+f^-(x_i,a_{i,j}^{\mathrm c})\frac{v(x_{i},t_{j}) - v(x_{i-1},t_j)}{\Delta x} \\
       & \quad\quad - f(x_i,a_{i,j}^{\mathrm c})\partial_x v(x_i,t_j).
    \end{split}
  \end{align}
  The inequality in \eqref{eq:eps_1} follows from the minimizing property of $A^*_{i,j}$ in the discrete problem \eqref{eq:upwind_HJB_input}, by evaluating the discrete objective at $a_{i,j}^{\mathrm c}$.
  This is equivalently written as
  \begin{align}
    \begin{split}
      \epsilon_{i,j-1} & \le (1-\alpha |f(x_i,a_{i,j}^{\mathrm c})|)\epsilon_{i,j} \\
                       & \quad
      +\alpha f^+(x_i,a_{i,j}^{\mathrm c})\epsilon_{i+1,j}
      -\alpha f^-(x_i,a_{i,j}^{\mathrm c})\epsilon_{i-1,j}
      +\Delta t (r_{i,j}^{1}+r_{i,j}^{2,\mathrm c}).
    \end{split}
  \end{align}
  Consequently, setting $e_j\coloneqq\max_{i}|\epsilon_{i,j}|$ and using the CFL condition \eqref{eq:CFL},
  which guarantees that the coefficients $1-\alpha|f(x_i,a_{i,j}^{\mathrm c})|$, $\alpha f^+(x_i,a_{i,j}^{\mathrm c})$, and $-\alpha f^-(x_i,a_{i,j}^{\mathrm c})$
  are all nonnegative and sum to one, we have
  \begin{align}
    \begin{split}
      \epsilon_{i,j-1}
       & \le \left| 1-\alpha |f(x_i,a_{i,j}^{\mathrm c})|\right||\epsilon_{i,j}| \\
       & \quad
      +\alpha f^+(x_i,a_{i,j}^{\mathrm c})|\epsilon_{i+1,j}|
      -\alpha f^-(x_i,a_{i,j}^{\mathrm c})|\epsilon_{i-1,j}|
      +\Delta t (r_{i,j}^{1}+r_{i,j}^{2,\mathrm c})                              \\
       & \le e_{j}+\Delta t (r_{i,j}^{1}+r_{i,j}^{2,\mathrm c}).
    \end{split}\label{eq:plus_eps}
  \end{align}
  Similarly, since $a_{i,j}^{\mathrm c}$ minimizes the continuous objective defining $H(x_i,\partial_x v(x_i,t_j))$, evaluating that objective at $A^*_{i,j}$ yields
  \begin{align}
    \begin{split}
       & \frac{\epsilon_{i,j-1} - \epsilon_{i,j}}{\Delta t}-f^{+}\left(x_i,A_{i,j}^*\right) \frac{\epsilon_{i+1,j}-\epsilon_{i,j}}{\Delta x}-f^{-}\left(x_i,A_{i,j}^*\right) \frac{\epsilon_{i,j}-\epsilon_{i-1,j}}{\Delta x} \\
       & \ge
      \frac{v(x_i,t_{j}) - v(x_i,t_{j-1})}{\Delta t} -
      \partial_t v(x_i,t_j)                                                                                                                                                                                                   \\
       & \quad
      + f^+(x_i,A_{i,j}^*)\frac{v(x_{i+1},t_{j}) - v(x_i,t_j)}{\Delta x}
      + f^-(x_i,A_{i,j}^*)\frac{v(x_{i},t_{j}) - v(x_{i-1},t_j)}{\Delta x}                                                                                                                                                    \\
       & \quad
      - f(x_i,A_{i,j}^*)\partial_x v(x_i,t_j)                                                                                                                                                                                 \\
       & = r_{i,j}^{1} + r_{i,j}^{2,\mathrm d},
    \end{split}
  \end{align}
  where we have defined
  \begin{align}
    \begin{split}
      r_{i,j}^{2,\mathrm d} & \coloneqq f^+(x_i,A_{i,j}^*)\frac{v(x_{i+1},t_{j}) - v(x_i,t_j)}{\Delta x}
      +f^-(x_i,A_{i,j}^*)\frac{v(x_{i},t_{j}) - v(x_{i-1},t_j)}{\Delta x}                                \\
                            & \quad -
      f(x_i,A_{i,j}^*)\partial_x v(x_i,t_j),
    \end{split}
  \end{align}
  from which, again by the nonnegativity of the coefficients ensured by \eqref{eq:CFL}, we obtain
  \begin{align}
    \begin{split}
      -\epsilon_{i,j-1}
       & \le \left| 1-\alpha |f(x_i,A_{i,j}^*)| \right| |\epsilon_{i,j}| \\
       & \quad
      +\alpha f^+(x_i,A_{i,j}^*)|\epsilon_{i+1,j}|
      -\alpha f^-(x_i,A_{i,j}^*)|\epsilon_{i-1,j}|
      -\Delta t (r_{i,j}^{1}+r_{i,j}^{2,\mathrm d})                      \\
       & \le e_{j} - \Delta t (r_{i,j}^{1}+r_{i,j}^{2,\mathrm d}).
    \end{split}\label{eq:minus_eps}
  \end{align}
  By the periodic H\"{o}lder regularity assumed in Theorem~\ref{thm:convergence}, the standard error estimates for difference quotients give
  \begin{equation}\label{eq:residual_bounds}
    |r_{i,j}^{1}| \le (\Delta t)^\eta \Lambda_t,\quad
    |r_{i,j}^{2,\mathrm c}| \le (\Delta x)^\eta \Lambda_x,\quad
    |r_{i,j}^{2,\mathrm d}| \le (\Delta x)^\eta \Lambda_x,
  \end{equation}
  where
  \begin{align}
    \Lambda_t & \coloneqq  \sup_{x\in\Omega}\ \sup_{\substack{t_1,t_2\in[0,T] \\ t_1\ne t_2}} \frac{|\partial_t v(x,t_1) - \partial_t v(x,t_2)|}{|t_1 - t_2|^\eta}, \label{eq:2drv_bound_t}                                \\
    \Lambda_x & \coloneqq  \sup_{a\in E}\ \sup_{t\in[0,T]}\ \sup_{\substack{x_1,x_2\in\Omega \\ x_1\ne x_2}} |f(x_1,a)|\frac{|\partial_x v(x_1,t) - \partial_x v(x_2,t)|}{d_\Omega(x_1,x_2)^\eta}. \label{eq:2drv_bound_x}
  \end{align}
  Both quantities are finite by the H\"{o}lder continuity that Theorem~\ref{thm:convergence} assumes for the extended derivatives on $\Omega\times[0,T]$.
  In the definition of $\Lambda_x$, the spatial arguments of $v$ are understood periodically on $\Omega$.
  Combining \eqref{eq:residual_bounds} with \eqref{eq:plus_eps} and \eqref{eq:minus_eps}, we obtain
  \begin{align}\label{eq:error_recurrence}
    e_{j-1}\le e_j + \Delta t\left( (\Delta x)^\eta \Lambda_x + (\Delta t)^\eta \Lambda_t\right).
  \end{align}
  For any $q=0,\ldots,N_t$, iterating \eqref{eq:error_recurrence} from $j=N_t$ down to $j=q+1$ gives
  \begin{align}
    \begin{split}
      e_q
       & \le e_{N_t}+(N_t-q)\Delta t
      \left((\Delta t)^\eta \Lambda_t+(\Delta x)^\eta \Lambda_x\right) \\
       & \le e_{N_t}+T
      \left((\Delta t)^\eta \Lambda_t+(\Delta x)^\eta \Lambda_x\right).
    \end{split}
  \end{align}
  Since the terminal values are imposed exactly, $e_{N_t}=0$.
  Define
  \begin{align}
    C_t\coloneqq T\Lambda_t,
    \qquad
    C_x\coloneqq T\Lambda_x.
  \end{align}
  Consequently,
  \begin{align}
    \max_{0\le q\le N_t}e_q
    \le C_t(\Delta t)^\eta+C_x(\Delta x)^\eta,
  \end{align}
  which proves the asserted uniform estimate and completes the proof.
\end{proof}

\subsection{Convergence of the One-Sided Difference Quotients}\label{sec:proof-gradients}

The proof of Proposition~\ref{prop:convergence_derivative} has three stages.
First, discrete semiconcavity and periodicity give uniform bounds on the numerical slopes and their spatial total variation.
Second, the slope bound and value convergence imply spatial Lipschitz regularity of $v$.
Third, discrete spatial averaging compares the numerical slopes with difference quotients over a fixed positive width; value convergence and the limit as this width tends to zero give strong $L^1$ convergence.
The two lemmas below provide the spatial bounds; we then prove the proposition.

For a fixed periodic mesh, we continue to suppress its mesh index.
We use the original time index $j=0,\ldots,N_t$ and the terminal condition $V_{i,N_t}=v_T(x_i)$.
We use the numerical Hamiltonian $\widehat H$ defined in \eqref{eq:numerical-Hamiltonian}.
Whenever a stencil crosses the boundary of the representative interval $[-1,1)$, the functions $f$, $g$, and $v_T$ are evaluated through their $2$-periodic lifts.
Define the forward slopes at each time level by
\begin{align}\label{eq:slopes-def}
  p_{i,j}\coloneqq D^+V_{i,j}
  =\frac{V_{i+1,j}-V_{i,j}}{\Delta x}.
\end{align}
All spatial indices are understood modulo $N_x$.
The scheme advances backward from time level $j$ to $j-1$:
\begin{align}\label{eq:backward-update}
  V_{i,j-1}=V_{i,j}+\Delta t \widehat H(x_i,p_{i-1,j},p_{i,j}),
\end{align}
for $j=1,\ldots,N_t$.

We denote by $p_j\colon\Omega\to\bbR$ the periodic piecewise constant function defined by $p_j(x)=p_{i,j}$ whenever $\pi_k(x)=x_i$.
For any periodic piecewise constant function $r$ on $\Omega$ with cell values $r_i$, we write
\begin{align}\label{eq:TV-def}
  \operatorname{TV}(r)\coloneqq\sum_i|r_i-r_{i-1}|,
\end{align}
which is its total variation on $\Omega$.
For every $\xi\in[-1,1]$, such a function satisfies
\begin{align}\label{eq:pc-translation}
  \|r(\cdot+\xi)-r\|_{L^1(\Omega)}
  \le |\xi|\operatorname{TV}(r).
\end{align}
Indeed, each jump contributes to the translated difference only on an interval of length at most $|\xi|$, and summing over the jumps gives \eqref{eq:pc-translation}.
Finally, set
\begin{align}\label{eq:G-const}
  G_1\coloneqq
  \sup_{(x,a)\in\Omega\times E}|\partial_x g(x,a)|.
\end{align}

We first prove a semiconcavity bound for the numerical solution with a constant independent of the mesh.
This bound will be used to control the numerical slopes and their spatial variation.

\begin{lemma}[Uniform discrete semiconcavity]\label{lmm:discrete_semiconcavity}
  Suppose that Assumptions~\ref{assmp:HJB} and \ref{assmp:deriv} hold and that the fixed periodic mesh satisfies
  \begin{align}\label{eq:mesh-CFL}
    \frac{\Delta t}{\Delta x}F_0+\Delta t F_1\le1.
  \end{align}
  Set
  \begin{align}
    \begin{split}
      F_2             & \coloneqq\sup_{(x,a)\in\Omega\times E}|\partial_{xx} f(x,a)|, \\
      G_2             & \coloneqq\sup_{(x,a)\in\Omega\times E}|\partial_{xx} g(x,a)|, \\
      C_{\mathrm{sc}} & \coloneqq
      \exp\bigl((3F_1+2F_2)T\bigr)
      \left(\|\partial_{xx} v_T\|_{L^\infty(\Omega)}+TG_2\right).
    \end{split}
  \end{align}
  Then the numerical solution satisfies
  \begin{align}\label{eq:discrete-semiconcavity-bound}
    \frac{V_{i+1,j}-2V_{i,j}+V_{i-1,j}}{(\Delta x)^2}\le C_{\mathrm{sc}},
  \end{align}
  for every spatial index $i$ and every time index $j$.
  The constant $C_{\mathrm{sc}}$ is independent of $\Delta x$ and $\Delta t$.
\end{lemma}

\begin{proof}
  The proof is given in Appendix~\ref{app:semiconcavity}.
\end{proof}

The left-hand side of \eqref{eq:discrete-semiconcavity-bound} in Lemma~\ref{lmm:discrete_semiconcavity} is the centered second spatial difference, and its mesh-independent upper bound is what we call \emph{discrete semiconcavity}.
The semiconcavity estimate of Lemma~\ref{lmm:discrete_semiconcavity} bounds the differences between neighboring numerical slopes from above.
Combined with periodicity, it nevertheless yields the two-sided bounds needed for the spatial averaging argument.

\begin{lemma}[Uniform slope and spatial variation bounds]\label{lmm:slope_variation}
  Suppose that Assumptions~\ref{assmp:HJB} and \ref{assmp:deriv} hold and that the fixed periodic mesh satisfies \eqref{eq:mesh-CFL}.
  Let $C_{\mathrm{sc}}$ be the mesh-independent constant of Lemma~\ref{lmm:discrete_semiconcavity}.
  Then, at every time level $j=0,\ldots,N_t$,
  \begin{align}\label{eq:uniform-slope-tv}
    \|p_j\|_{L^\infty(\Omega)}\le2C_{\mathrm{sc}},
    \qquad
    \operatorname{TV}(p_j)\le4C_{\mathrm{sc}}.
  \end{align}
\end{lemma}

\begin{proof}
  We derive bounds, uniform in the mesh and time level, for the slopes and their spatial total variation from discrete semiconcavity and periodicity.
  Write
  \begin{align}
    d_{i,j}\coloneqq p_{i,j}-p_{i-1,j}
  \end{align}
  for the increments of the slopes.
  By the definition \eqref{eq:slopes-def} of $p_{i,j}$,
  \begin{align}
    \frac{d_{i,j}}{\Delta x}
    =\frac{V_{i+1,j}-2V_{i,j}+V_{i-1,j}}{(\Delta x)^2},
  \end{align}
  which is bounded above by $C_{\mathrm{sc}}$ according to \eqref{eq:discrete-semiconcavity-bound} in Lemma~\ref{lmm:discrete_semiconcavity}.
  Hence
  \begin{align}\label{eq:slope-step-T}
    d_{i,j}\le C_{\mathrm{sc}}\Delta x,
  \end{align}
  for every $i$ and $j$.
  Since the grid values $V_{i,j}$ are periodic in $i$, their forward differences telescope over one period, so that
  \begin{align}\label{eq:telescope-T}
    \sum_i p_{i,j}=0,
    \qquad
    \sum_i d_{i,j}=0.
  \end{align}
  From \eqref{eq:slope-step-T} and $N_x\Delta x=2$,
  \begin{align}
    \sum_i(d_{i,j})^+\le N_x\Delta x C_{\mathrm{sc}}=2C_{\mathrm{sc}},
  \end{align}
  and the second identity in \eqref{eq:telescope-T} shows that the sum of the positive parts equals the absolute value of the sum of the negative parts, whence
  \begin{align}\label{eq:tv-bound}
    \sum_i|d_{i,j}|
    =2\sum_i(d_{i,j})^+
    \le4C_{\mathrm{sc}}.
  \end{align}
  By the first identity in \eqref{eq:telescope-T}, $\min_i p_{i,j}\le0\le\max_i p_{i,j}$.
  Moving in the positive periodic direction from a minimizing index to a maximizing index gives
  \begin{align*}
    \max_i p_{i,j}-\min_i p_{i,j}
    \le\sum_i(d_{i,j})^+
    \le2C_{\mathrm{sc}}.
  \end{align*}
  Therefore
  \begin{align}\label{eq:p-bound-T}
    |p_{i,j}|\le2C_{\mathrm{sc}}.
  \end{align}
  Estimates~\eqref{eq:p-bound-T} and \eqref{eq:tv-bound} establish the uniform slope and total-variation bounds \eqref{eq:uniform-slope-tv}.
\end{proof}

We now prove Proposition~\ref{prop:convergence_derivative} from the slope and spatial variation bounds of Lemma~\ref{lmm:slope_variation} and the value convergence of Proposition~\ref{prop:value_uniform_convergence}.

\begin{proof}[Proof of Proposition~\ref{prop:convergence_derivative}]
  We first establish spatial Lipschitz regularity of $v$, then prove strong $L^1$ convergence of $D^+V_k$ to $\partial_xv$, and finally transfer the convergence to $D^-V_k$.

  We work on the tail where \eqref{eq:semiconcavity_CFL} holds.
  This condition implies \eqref{eq:value-convergence-CFL}, so Proposition~\ref{prop:value_uniform_convergence} gives \eqref{eq:value-uniform-convergence}.
  By \eqref{eq:uniform-slope-tv} in Lemma~\ref{lmm:slope_variation}, $\|D^+V_k(\cdot,t)\|_{L^\infty(\Omega)}\le2C_{\mathrm{sc}}$ and $\operatorname{TV}(D^+V_k(\cdot,t))\le4C_{\mathrm{sc}}$ for every $t\in[0,T]$.

  \emph{Step 1: Spatial Lipschitz regularity.}
  We first establish a spatial Lipschitz bound for $v$ at every time, which gives the weak derivative needed below.
  Fix $x,y\in\Omega$ and $t\in[0,T]$.
  The mesh and the projection are already fixed, so the choice of $x,y$ determines $\pi_k(x)$ and $\pi_k(y)$.
  Both $|V_k(x,t)-V_k(y,t)|$ and $d_\Omega(\pi_k(x),\pi_k(y))$ are unchanged by interchanging $x$ and $y$.
  We may therefore order the two chosen points so that $\pi_k(x)=x_i$ and $\pi_k(y)=x_{i+m}$ for periodic grid indices $i,i+m$ with $0\le m\le N_{x,k}/2$.
  Then $d_\Omega(\pi_k(x),\pi_k(y))=m\Delta x_k$.
  Since $V_k$ is constant on each projected cell, the bound $|D^+V_k(x_i,t)|\le2C_{\mathrm{sc}}$ from \eqref{eq:uniform-slope-tv} gives
  \begin{align}\label{eq:projected-slope-bound}
    \begin{split}
      |V_k(x,t)-V_k(y,t)|
      &=\left|\sum_{\ell=0}^{m-1}\bigl(V_k(x_{i+\ell+1},t)-V_k(x_{i+\ell},t)\bigr)\right| \\
      &=\Delta x_k\left|\sum_{\ell=0}^{m-1}D^+V_k(x_{i+\ell},t)\right| \\
      &\le\Delta x_k\sum_{\ell=0}^{m-1}|D^+V_k(x_{i+\ell},t)| \\
      &\le2C_{\mathrm{sc}}m\Delta x_k
      =2C_{\mathrm{sc}}d_\Omega(\pi_k(x),\pi_k(y)).
    \end{split}
  \end{align}
  The triangle inequality and the projection bound \eqref{eq:grid-projection-distance} give
  \begin{align}\label{eq:projected-distance}
    \begin{split}
      d_\Omega(\pi_k(x),\pi_k(y))
      &\le d_\Omega(\pi_k(x),x)+d_\Omega(x,y)+d_\Omega(y,\pi_k(y)) \\
      &\le\frac{\Delta x_k}{2}+d_\Omega(x,y)+\frac{\Delta x_k}{2} \\
      &=d_\Omega(x,y)+\Delta x_k.
    \end{split}
  \end{align}
  Combining \eqref{eq:projected-slope-bound} and \eqref{eq:projected-distance} yields
  \begin{align}\label{eq:value-limit-error}
    \begin{split}
      |v(x,t)-v(y,t)|
      &\le |v(x,t)-V_k(x,t)|+|V_k(x,t)-V_k(y,t)| \\
      &\quad+|V_k(y,t)-v(y,t)| \\
      &\le2\sup_{\Omega\times[0,T]}|V_k-v|
      +2C_{\mathrm{sc}}\bigl(d_\Omega(x,y)+\Delta x_k\bigr).
    \end{split}
  \end{align}
  By \eqref{eq:value-uniform-convergence}, the supremum in \eqref{eq:value-limit-error} tends to zero.
  Letting $k\to\infty$ and using $\Delta x_k\to0$ therefore gives
  \begin{align}\label{eq:v-lipschitz}
    |v(x,t)-v(y,t)|\le 2C_{\mathrm{sc}}d_\Omega(x,y)
  \end{align}
  for all $x,y\in\Omega$ and $t\in[0,T]$.
  Thus $v$ has a weak spatial derivative in $L^\infty(Q)$, where $Q=\Omega\times(0,T)$, agreeing almost everywhere with its pointwise spatial derivative.

  \emph{Step 2: Strong convergence of $D^+V_k$.}
  To prove strong convergence, we average $D^+V_k$ over grid shifts within a positive spatial width $r$.
  Fix $0<r\le1$ and set $m_k=\lfloor r/\Delta x_k\rfloor$ and $r_k=m_k\Delta x_k$.
  The definition of the integer part gives
  \begin{align}\label{eq:r-grid-rounding}
    m_k\Delta x_k\le r<(m_k+1)\Delta x_k,
    \qquad 0\le r-r_k<\Delta x_k.
  \end{align}
  Hence $r_k\to r$ as $k\to\infty$ with $r$ fixed.
  In particular, $m_k\ge1$ and $r_k>0$ for all sufficiently large $k$.
  The triangle inequality decomposes the error into three terms:
  \begin{align}\label{eq:gradient-three-errors}
    \begin{split}
      &\|D^+V_k-\partial_xv\|_{L^1(Q)} \\
      &\quad\le\left\|D^+V_k-\frac{V_k(\cdot+r_k,\cdot)-V_k}{r_k}\right\|_{L^1(Q)} \\
      &\qquad+\left\|\frac{V_k(\cdot+r_k,\cdot)-V_k}{r_k}
      -\frac{v(\cdot+r,\cdot)-v}{r}\right\|_{L^1(Q)} \\
      &\qquad+\left\|\frac{v(\cdot+r,\cdot)-v}{r}-\partial_xv\right\|_{L^1(Q)}.
    \end{split}
  \end{align}
  We will bound the first term on the right-hand side of \eqref{eq:gradient-three-errors} by $2C_{\mathrm{sc}}Tr$, show that the second tends to zero as $k\to\infty$ with $r$ fixed, and show that the third tends to zero as $r\downarrow0$.

  For the first term on the right-hand side of \eqref{eq:gradient-three-errors}, we use the spatial variation bound through an average of $D^+V_k$.
  Translation by $\ell\Delta x_k$ maps each half-open cell to the cell $\ell$ indices ahead.
  The definitions of $V_k$ and $D^+V_k$ therefore give
  \begin{align}\label{eq:sum-spatial-increments}
    \begin{split}
      &\Delta x_k\sum_{\ell=0}^{m_k-1}D^+V_k(x+\ell\Delta x_k,t) \\
      &\quad=\sum_{\ell=0}^{m_k-1}
      \bigl(V_k(x+(\ell+1)\Delta x_k,t)-V_k(x+\ell\Delta x_k,t)\bigr) \\
      &\quad=V_k(x+m_k\Delta x_k,t)-V_k(x,t) \\
      &\quad=V_k(x+r_k,t)-V_k(x,t).
    \end{split}
  \end{align}
  In the second equality, every intermediate value occurs once with each sign and cancels.
  Dividing \eqref{eq:sum-spatial-increments} by $r_k=m_k\Delta x_k$ yields
  \begin{align}\label{eq:coarse-slope-average}
    \frac{V_k(x+r_k,t)-V_k(x,t)}{r_k}
    =\frac1{m_k}\sum_{\ell=0}^{m_k-1}D^+V_k(x+\ell\Delta x_k,t).
  \end{align}
  Using \eqref{eq:coarse-slope-average} and the triangle inequality, we obtain
  \begin{align}\label{eq:coarse-slope-triangle}
    \begin{split}
      &\left\|D^+V_k-\frac{V_k(\cdot+r_k,\cdot)-V_k}{r_k}\right\|_{L^1(Q)} \\
      &\quad=\left\|\frac1{m_k}\sum_{\ell=0}^{m_k-1}
      \bigl(D^+V_k-D^+V_k(\cdot+\ell\Delta x_k,\cdot)\bigr)\right\|_{L^1(Q)} \\
      &\quad\le\frac1{m_k}\sum_{\ell=0}^{m_k-1}\int_0^T
      \|D^+V_k(\cdot,t)-D^+V_k(\cdot+\ell\Delta x_k,t)\|_{L^1(\Omega)}\,\rmd t.
    \end{split}
  \end{align}
  For $0\le\ell\le m_k-1$, the shift satisfies $0\le\ell\Delta x_k\le r\le1$.
  Thus \eqref{eq:pc-translation} and \eqref{eq:uniform-slope-tv} give, at every time,
  \begin{align}\label{eq:translated-slope-bound}
    \begin{split}
      \|D^+V_k(\cdot,t)-D^+V_k(\cdot+\ell\Delta x_k,t)\|_{L^1(\Omega)}
      &\le\ell\Delta x_k\operatorname{TV}(D^+V_k(\cdot,t)) \\
      &\le4C_{\mathrm{sc}}\ell\Delta x_k.
    \end{split}
  \end{align}
  Substituting \eqref{eq:translated-slope-bound} into \eqref{eq:coarse-slope-triangle}, integrating in time, and summing the integers from $0$ to $m_k-1$ gives
  \begin{align}\label{eq:coarse-slope-error}
    \begin{split}
      \left\|D^+V_k-\frac{V_k(\cdot+r_k,\cdot)-V_k}{r_k}\right\|_{L^1(Q)}
      &\le\frac{4C_{\mathrm{sc}}T\Delta x_k}{m_k}\sum_{\ell=0}^{m_k-1}\ell \\
      &=\frac{4C_{\mathrm{sc}}T\Delta x_k}{m_k}\frac{m_k(m_k-1)}2 \\
      &=2C_{\mathrm{sc}}T(m_k-1)\Delta x_k \\
      &\le2C_{\mathrm{sc}}Tr,
    \end{split}
  \end{align}
  where the last inequality follows from $(m_k-1)\Delta x_k\le r_k\le r$.

  For the second term on the right-hand side of \eqref{eq:gradient-three-errors}, we compare the two difference quotients using value convergence.
  For this fixed $r>0$, \eqref{eq:v-lipschitz} also gives, for every $x,t$,
  \begin{align}\label{eq:difference-quotient-limit}
    \begin{split}
      &\left|\frac{V_k(x+r_k,t)-V_k(x,t)}{r_k}
      -\frac{v(x+r,t)-v(x,t)}r\right| \\
      &\quad\le\frac{|V_k(x+r_k,t)-v(x+r_k,t)|+|V_k(x,t)-v(x,t)|}{r_k} \\
      &\qquad+\frac{|v(x+r_k,t)-v(x+r,t)|}{r_k}
      +\left|\frac1{r_k}-\frac1r\right||v(x+r,t)-v(x,t)| \\
      &\quad\le\frac2{r_k}\sup_{\Omega\times[0,T]}|V_k-v|
      +\frac{2C_{\mathrm{sc}}|r_k-r|}{r_k}
      +\frac{|r_k-r|}{r_kr}\,2C_{\mathrm{sc}}r \\
      &\quad=\frac2{r_k}\sup_{\Omega\times[0,T]}|V_k-v|
      +\frac{4C_{\mathrm{sc}}|r_k-r|}{r_k}.
    \end{split}
  \end{align}
  By \eqref{eq:r-grid-rounding} and \eqref{eq:value-uniform-convergence}, the right-hand side of \eqref{eq:difference-quotient-limit} tends to zero as $k\to\infty$ and is independent of $x,t$.
  Using \eqref{eq:coarse-slope-error} and \eqref{eq:difference-quotient-limit} in \eqref{eq:gradient-three-errors}, and taking $k\to\infty$ with $r$ fixed, yields
  \begin{align}\label{eq:direct-gradient-limsup}
    \begin{split}
      \limsup_{k\to\infty}\|D^+V_k-\partial_xv\|_{L^1(Q)}
      &\le 2C_{\mathrm{sc}}Tr \\
      &\quad+\left\|\frac{v(\cdot+r,\cdot)-v}{r}-\partial_xv\right\|_{L^1(Q)}.
    \end{split}
  \end{align}
  We now let $r\downarrow0$.
  The term $2C_{\mathrm{sc}}Tr$ in \eqref{eq:direct-gradient-limsup} tends to zero; we handle the third term from \eqref{eq:gradient-three-errors} using spatial translation continuity.
  Since $v(\cdot,t)$ is Lipschitz, the fundamental theorem of calculus and the triangle inequality give
  \begin{align}\label{eq:limit-derivative-difference}
    \begin{split}
      &\left\|\frac{v(\cdot+r,\cdot)-v}{r}-\partial_xv\right\|_{L^1(Q)} \\
      &\quad=\left\|\frac1r\int_0^r
      \bigl(\partial_xv(\cdot+s,\cdot)-\partial_xv\bigr)\,\rmd s\right\|_{L^1(Q)} \\
      &\quad\le\frac1r\int_0^r
      \|\partial_xv(\cdot+s,\cdot)-\partial_xv\|_{L^1(Q)}\,\rmd s
      \longrightarrow0,
    \end{split}
  \end{align}
  where the limit as $r\downarrow0$ follows from continuity of spatial translations in $L^1(Q)$.
  The left-hand side of \eqref{eq:direct-gradient-limsup} is nonnegative and independent of $r$.
  Since \eqref{eq:direct-gradient-limsup} holds for every $0<r\le1$, letting $r\downarrow0$ and using \eqref{eq:limit-derivative-difference} shows that this left-hand side is zero.
  Hence the full sequence $D^+V_k$ converges strongly to $\partial_xv$ in $L^1(Q)$.

  \emph{Step 3: Strong convergence of $D^-V_k$.}
  We transfer the convergence of $D^+V_k$ to $D^-V_k$ by showing that their $L^1(Q)$ difference tends to zero.
  Finally, $D^-V_k(x,t)=D^+V_k(x-\Delta x_k,t)$, so \eqref{eq:pc-translation} and \eqref{eq:uniform-slope-tv} yield
  \begin{align}\label{eq:one-sided-slope-difference}
    \begin{split}
      \|D^-V_k-D^+V_k\|_{L^1(Q)}
      &=\int_0^T\|D^+V_k(\cdot-\Delta x_k,t)-D^+V_k(\cdot,t)\|_{L^1(\Omega)}\,\rmd t \\
      &\le\Delta x_k\int_0^T\operatorname{TV}(D^+V_k(\cdot,t))\,\rmd t \\
      &\le4C_{\mathrm{sc}}T\Delta x_k\longrightarrow0.
    \end{split}
  \end{align}
  The triangle inequality gives the asserted convergence of $D^-V_k$.
\end{proof}

\subsection{Convergence of the Feedback Control}\label{sec:proof-controls}

The argument first passes from strong $L^1$ convergence of both difference quotients to their simultaneous almost-everywhere convergence along a common subsequence.
At each point of a common full-measure set, this yields uniform convergence of the discrete minimization objectives on the control set.
We then apply the following lemma to the discrete minimizers.

\begin{lemma}[Convergence of minimizers]\label{lmm:uniform_minimizers}
  Let $E\subset\bbR^m$ be nonempty and compact, and let $\Phi_n,\Phi\in\calC(E)$ satisfy
  \begin{align}
    \max_{a\in E}|\Phi_n(a)-\Phi(a)|\longrightarrow0.
  \end{align}
  That is, $\Phi_n\to\Phi$ uniformly on $E$.
  Suppose that $\argmin_{a\in E}\Phi(a)=\{a^*\}$.
  Then every sequence $a_n\in\argmin_{a\in E}\Phi_n(a)$ satisfies $a_n\to a^*$.
\end{lemma}

\begin{proof}
  Suppose that the conclusion fails.
  Then there are $\varepsilon>0$ and a subsequence, still denoted by $a_n$, such that $|a_n-a^*|\ge\varepsilon$ for every $n$.
  By compactness of $E$, a further subsequence converges to some $\bar a\in E$, and necessarily $|\bar a-a^*|\ge\varepsilon$.
  Relabel this further subsequence as $a_n$.
  The minimizing property gives
  \begin{align}
    \Phi_n(a_n)\le\Phi_n(a^*).
  \end{align}
  Uniform convergence implies
  \begin{align}
    |\Phi_n(a_n)-\Phi(a_n)|\longrightarrow0,
    \qquad
    \Phi_n(a^*)\longrightarrow\Phi(a^*),
  \end{align}
  and continuity of $\Phi$ gives $\Phi(a_n)\to\Phi(\bar a)$.
  Passing to the limit in the minimizing inequality yields
  \begin{align}
    \Phi(\bar a)\le\Phi(a^*).
  \end{align}
  Uniqueness of the minimizer forces $\bar a=a^*$, a contradiction.
  Hence $a_n\to a^*$.
\end{proof}

We now combine Proposition~\ref{prop:convergence_derivative} and Lemma~\ref{lmm:uniform_minimizers} to prove the control convergence asserted in Theorem~\ref{thm:convenience_input}.

\begin{proof}[Proof of Theorem~\ref{thm:convenience_input}]
  Since $v$ is spatially Lipschitz by Proposition~\ref{prop:convergence_derivative}, there exists a set $S_1\subset\Omega\times(0,T)$ of full measure such that $v$ is differentiable in $x$ at every point of $S_1$.
  On $S_1$, this derivative agrees with the weak spatial derivative appearing in Proposition~\ref{prop:convergence_derivative}.
  Let $G_1$ be the constant defined in \eqref{eq:G-const}.
  Proposition~\ref{prop:convergence_derivative} gives
  \begin{align}\label{eq:control-proof-gradient-L1}
    \|D^+V_k-\partial_x v\|_{L^1(\Omega\times(0,T))}
    +\|D^-V_k-\partial_x v\|_{L^1(\Omega\times(0,T))}
    \longrightarrow0.
  \end{align}
  Define the nonnegative function
  \begin{align}
    e_k(x,t)
    \coloneqq
    |D^+V_k(x,t)-\partial_x v(x,t)|
    +|D^-V_k(x,t)-\partial_x v(x,t)|.
  \end{align}
  By \eqref{eq:control-proof-gradient-L1} and this definition,
  $\|e_k\|_{L^1(\Omega\times(0,T))}\to0$.
  We can therefore choose recursively a strictly increasing sequence of indices $k_1<k_2<\cdots$ such that
  \begin{align}\label{eq:error-subsequence-bound}
    \|e_{k_\nu}\|_{L^1(\Omega\times(0,T))}\le2^{-\nu}.
  \end{align}
  This selection depends only on the sequence of global $L^1$-errors $\{\|e_k\|_{L^1(\Omega\times(0,T))}\}_k$ and is therefore independent of $(x,t)$.
  By \eqref{eq:error-subsequence-bound},
  \begin{align}
    \sum_{\nu=1}^\infty
    \int_0^T\int_\Omega e_{k_\nu}(x,t)\,\rmd x\rmd t
    =\sum_{\nu=1}^\infty\|e_{k_\nu}\|_{L^1(\Omega\times(0,T))}
    \le\sum_{\nu=1}^\infty2^{-\nu}
    <\infty.
  \end{align}
  Tonelli's theorem implies
  \begin{align}
    \sum_{\nu=1}^\infty e_{k_\nu}(x,t)<\infty,
  \end{align}
  for almost every $(x,t)$.
  Define
  \[
    S\coloneqq\left\{(x,t)\in S_1:
    \sum_{\nu=1}^\infty e_{k_\nu}(x,t)<\infty\right\}.
  \]
  Since $S_1$ has full measure, so does $S$ in $\Omega\times(0,T)$ and hence also in $\Omega\times[0,T]$.
  Consequently, along this single subsequence,
  \begin{align}
    \begin{split}
      D^+V_{k_\nu}(x,t) & \longrightarrow\partial_x v(x,t), \\
      D^-V_{k_\nu}(x,t) & \longrightarrow\partial_x v(x,t),
    \end{split}
  \end{align}
  simultaneously at every $(x,t)\in S$.

  Fix an arbitrary point $(x,t)\in S$.
  Set
  \begin{align}
    \begin{split}
      p       & \coloneqq\partial_x v(x,t),  \\
      P_\nu^+ & \coloneqq D^+V_{k_\nu}(x,t), \\
      P_\nu^- & \coloneqq D^-V_{k_\nu}(x,t), \\
      x_\nu   & \coloneqq\pi_{k_\nu}(x).
    \end{split}
  \end{align}
  Then
  \begin{align}
    P_\nu^+\to p,
    \qquad P_\nu^-\to p,
    \qquad
    d_\Omega(x_\nu,x)\le\frac{\Delta x_{k_\nu}}{2}\to0.
  \end{align}
  At this fixed point, abbreviate
  \begin{align}
    \begin{split}
      \Phi_\nu(a)
       & \coloneqq
      f^+(x_\nu,a)P_\nu^+
      +f^-(x_\nu,a)P_\nu^-
      +g(x_\nu,a),                 \\
      \Phi(a)
       & \coloneqq f(x,a)p+g(x,a).
    \end{split}
  \end{align}
  Since $f=f^++f^-$, the objective difference has the exact decomposition
  \begin{align}
    \begin{split}
      \Phi_\nu(a)-\Phi(a)
       & =f^+(x_\nu,a)(P_\nu^+-p)
      +f^-(x_\nu,a)(P_\nu^--p)                            \\
       & \quad+[f(x_\nu,a)-f(x,a)]p
                                    +[g(x_\nu,a)-g(x,a)].
    \end{split}
  \end{align}
  Therefore, uniformly for $a\in E$,
  \begin{align}\label{eq:Phi_uniform_convergence}
    \begin{split}
      \max_{a\in E}|\Phi_\nu(a)-\Phi(a)|
       & \le
      F_0\left(|P_\nu^+-p|+|P_\nu^--p|\right)
      +(F_1|p|+G_1)d_\Omega(x_\nu,x) \\
       & \le
      F_0\left(|P_\nu^+-p|+|P_\nu^--p|\right)
      +\frac{\Delta x_{k_\nu}}{2}(F_1|p|+G_1)
      \longrightarrow0.
    \end{split}
  \end{align}
  Estimate~\eqref{eq:Phi_uniform_convergence} recombines $f^+$ and $f^-$ before taking the spatial variation and uses no regularity of an optimizing control.

  The functions $\Phi_\nu$ and $\Phi$ are continuous on $E$ because $f(x,\cdot)$ and $g(x,\cdot)$ are continuous and $f^\pm$ are continuous compositions.
  The set $E\subset\bbR^m$ is compact by Assumption~\ref{assmp:HJB}.
  By Assumption~\ref{assmp:control_uniqueness},
  \begin{align}
    \argmin_{a\in E}\Phi(a)
    =\{a^*(x,p)\}.
  \end{align}
  Moreover, the definition of the discrete control gives
  \begin{align}
    A_{k_\nu}^*(x,t)\in\argmin_{a\in E}\Phi_\nu(a).
  \end{align}
  Lemma~\ref{lmm:uniform_minimizers} now yields
  \begin{align}
    A_{k_\nu}^*(x,t)
    \longrightarrow
    a^*(x,p)
    =a^*(x,\partial_x v(x,t)).
  \end{align}
  Since the fixed point was arbitrary in $S$, this proves the asserted convergence at every point of $S$.
  The subsequence was chosen before fixing $(x,t)$, so it is independent of $(x,t)$.
  Even when the discrete minimizer is nonunique, all minimizers give the same minimum value; thus $V_{i,j}$ and $D^\pm V_k$ do not depend on which minimizer is selected.
  Thus both $\{k_\nu\}$ and $S$ can be chosen independently of that choice.
  Lemma~\ref{lmm:uniform_minimizers} applies to every selection of discrete minimizers and yields the stated convergence.
\end{proof}

\section{Numerical Experiments}\label{sec:numeric}

This section reports value-function and control errors for the upwind finite-difference scheme on two examples in the periodic setting of Section~\ref{sec:scheme}.
The first example is a smooth periodic problem with closed-form expressions for the exact value function and optimal control.
It satisfies all hypotheses used in Sections~\ref{sec:convergence} and \ref{sec:proof}.
We compare the observed value-function rate with Theorem~\ref{thm:convergence}; the control experiment complements Theorem~\ref{thm:convenience_input}.
The second is a sparse-control problem whose exact value function is again available in closed form but whose pointwise minimizer is nonunique on a region of positive measure, so that it probes the sensitivity of the selected control to the numerical difference quotients.
Throughout this section, $V_{\Delta x}$ denotes a numerical value function computed with spatial mesh width $\Delta x$ and time step $\Delta t=0.5\Delta x$, and $A_{\Delta x}^*$ denotes the associated discrete control.
For a grid function $z$, the discrete space-time norms used throughout this section are
\begin{align}\label{eq:discrete_norms}
  \|z\|_{\infty}\coloneqq\max_{i,j}|z(x_i,t_j)|,
  \qquad
  \|z\|_{2}\coloneqq\left(\sum_{i,j}|z(x_i,t_j)|^2\Delta x\Delta t\right)^{1/2}.
\end{align}
The first example is measured in $\|\cdot\|_{\infty}$, which corresponds to the maximum norm used in Theorem~\ref{thm:convergence}.
The second is measured in $\|\cdot\|_{2}$, which registers the size of the region on which two selected controls differ; there the difference stays of order one on that region as the mesh is refined.

\subsection{Smooth Periodic Example}\label{sec:numeric_1}

We first assess convergence in a smooth periodic problem, using exact solutions to measure the errors.
Fix $\kappa=0.2$ and $c=1$, and define
\begin{align}\label{eq:smooth_profile}
  u(x)\coloneqq\kappa\left(1-\cos(\pi x)\right).
\end{align}
Set $T=1$ and $E=[-1,1]$, and define
\begin{align}\label{eq:smooth_data}
  f(x,a)=a,
  \qquad
  g(x,a)=c+\frac{1}{2}|u'(x)|^2+\frac{1}{2}a^2,
  \qquad
  v_T=u.
\end{align}
Assumption~\ref{assmp:HJB} holds.
The set $E$ is compact, and $f$ and $g$ are continuous.
Since $f$ is independent of $x$, it satisfies the spatial Lipschitz condition; $v_T$ is continuous.
The data \eqref{eq:smooth_data} are a special case of Example~\ref{ex:quadratic_cost}, with
\begin{align}\label{eq:smooth_as_example}
  b=0,
  \qquad
  \beta=1,
  \qquad
  q=c+\frac{1}{2}|u'|^2,
  \qquad
  \rho=\frac{1}{2}.
\end{align}
The coefficients in \eqref{eq:smooth_as_example} are $\calC^\infty$ and $2$-periodic, with $\rho>0$.
Since $v_T=u$ is also $\calC^\infty$, Assumptions~\ref{assmp:deriv} and \ref{assmp:control_uniqueness} hold.

For every $p\in\bbR$ the unconstrained minimizer of $a\mapsto ap+\frac{1}{2}a^2$ is $a=-p$.
Since $|u'(x)|\le\kappa\pi<1$, evaluating at $p=u'(x)$ gives a minimizer in the interior of $E$, whence
\begin{align}\label{eq:smooth_hamiltonian}
  \min_{a\in E}\left[au'(x)+c+\frac{1}{2}|u'(x)|^2+\frac{1}{2}a^2\right]
  =-\frac{1}{2}|u'(x)|^2+c+\frac{1}{2}|u'(x)|^2
  =c,
\end{align}
with the minimum attained only at $a=-u'(x)$.
The unique minimizer that Example~\ref{ex:quadratic_cost} supplies is therefore interior to $E$ here, and is available in closed form.
The function
\begin{align}\label{eq:smooth_exact_v}
  v(x,t)=u(x)+c(T-t)
\end{align}
therefore satisfies $\partial_t v=-c$ and, by \eqref{eq:smooth_hamiltonian} applied at $p=\partial_x v(x,t)=u'(x)$, solves the HJB equation \eqref{eq:HJB} with terminal value $u$.
Its exact feedback control is
\begin{align}\label{eq:smooth_exact_a}
  a_{\mathrm{ex}}(x,t)
  \coloneqq a^*(x,\partial_x v(x,t))
  =-u'(x)=-\kappa\pi\sin(\pi x).
\end{align}
Because $\partial_t v$ is constant and $\partial_x v=u'$ is Lipschitz continuous, Theorem~\ref{thm:convergence} applies with $\eta=1$ and predicts a first-order rate.

At each time the value function \eqref{eq:smooth_exact_v} attains its spatial minimum at $x=0$, and the feedback \eqref{eq:smooth_exact_a} points towards that minimum: $a_{\mathrm{ex}}<0$ on $(0,1)$ and $a_{\mathrm{ex}}>0$ on $(-1,0)$.
Near the origin $1-\cos(\pi x)=\frac{\pi^2}{2}x^2+O(x^4)$ and $a_{\mathrm{ex}}(x,t)=-\kappa\pi^2x+O(x^3)$, so the example has a quadratic value function and a linear stabilizing feedback there.

Here $F_0=1$ and $F_1=0$, so the choice $\Delta t=0.5\Delta x$ gives $\alpha=1/2$ and satisfies both \eqref{eq:CFL} and \eqref{eq:semiconcavity_CFL}.
Up to a term independent of $a$, the discrete minimization \eqref{eq:upwind_HJB_input} reads
\begin{align}\label{eq:smooth_discrete_objective}
  \min_{a\in E}\left[a^+D^+V_{i,j}+a^-D^-V_{i,j}+\frac{1}{2}a^2\right].
\end{align}
The objective in \eqref{eq:smooth_discrete_objective} is convex and quadratic on each of $[0,1]$ and $[-1,0]$, so a global minimizer is found among three candidates: the projection of $-D^+V_{i,j}$ onto $[0,1]$, the projection of $-D^-V_{i,j}$ onto $[-1,0]$, and $a=0$.
We evaluate the objective at these three points and take the candidate with the smallest objective value, breaking ties deterministically.

\begin{figure}[th]
  \centering
  \includegraphics[width=0.8\textwidth]{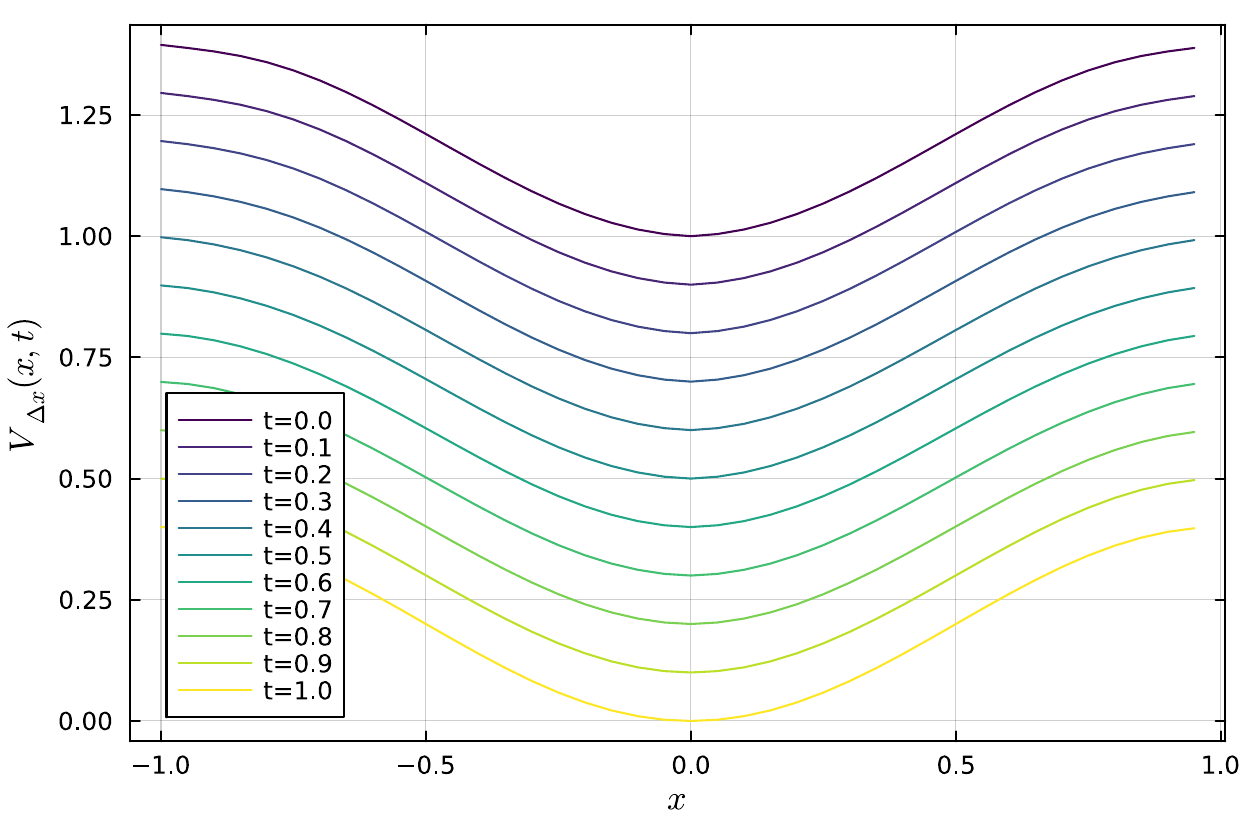}
  \caption{Time response of the numerical value function $V_{\Delta x}$.}
  \label{fig:HJB_v}
\end{figure}
\begin{figure}[th]
  \centering
  \includegraphics[width=0.8\textwidth]{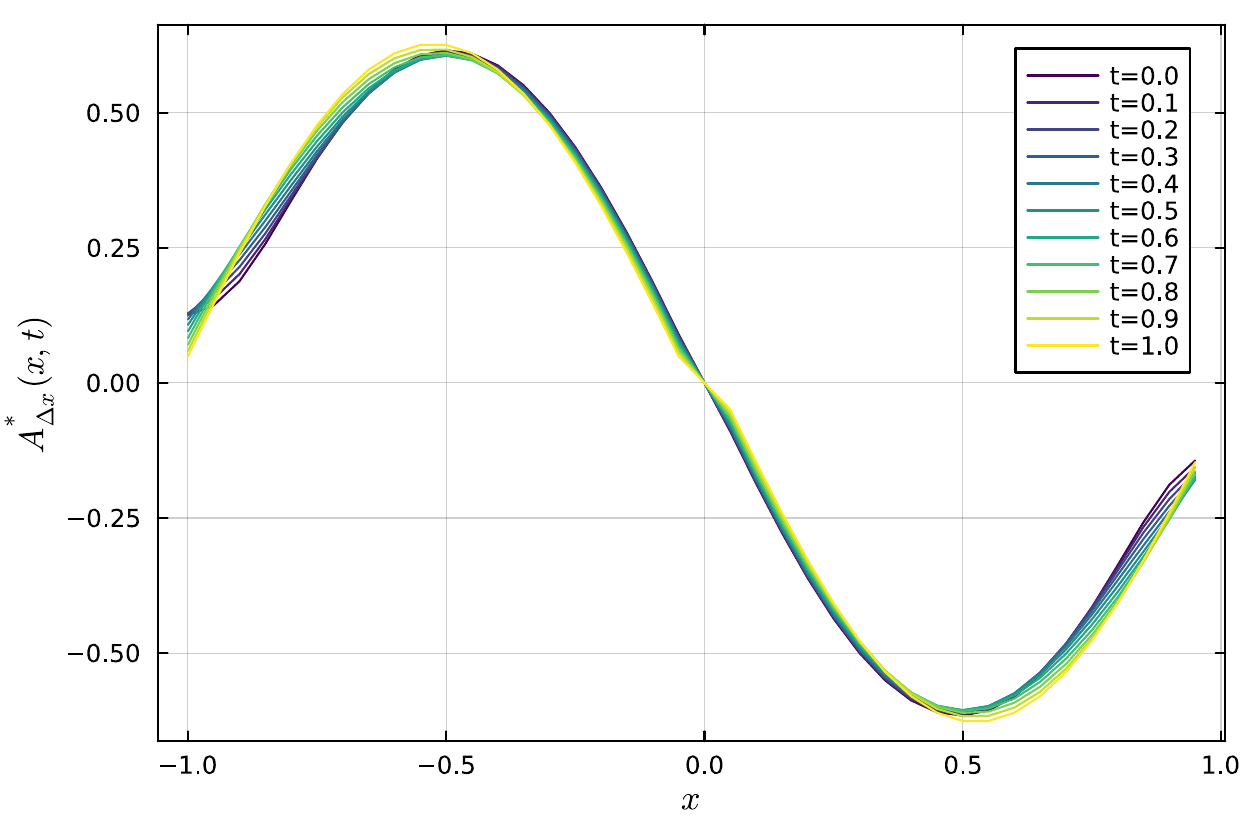}
  \caption{Time response of the numerical control $A_{\Delta x}^*$.}
  \label{fig:HJB_a}
\end{figure}

Figures~\ref{fig:HJB_v} and \ref{fig:HJB_a} show the time responses of $V_{\Delta x}$ and $A_{\Delta x}^*$ for $\Delta x=0.05$.
The computed value function reproduces the vertical translation of $u$ described by \eqref{eq:smooth_exact_v}, and the computed control reproduces the profile \eqref{eq:smooth_exact_a}, which is independent of $t$.

\begin{figure}[th]
  \centering
  \includegraphics[width=0.8\textwidth]{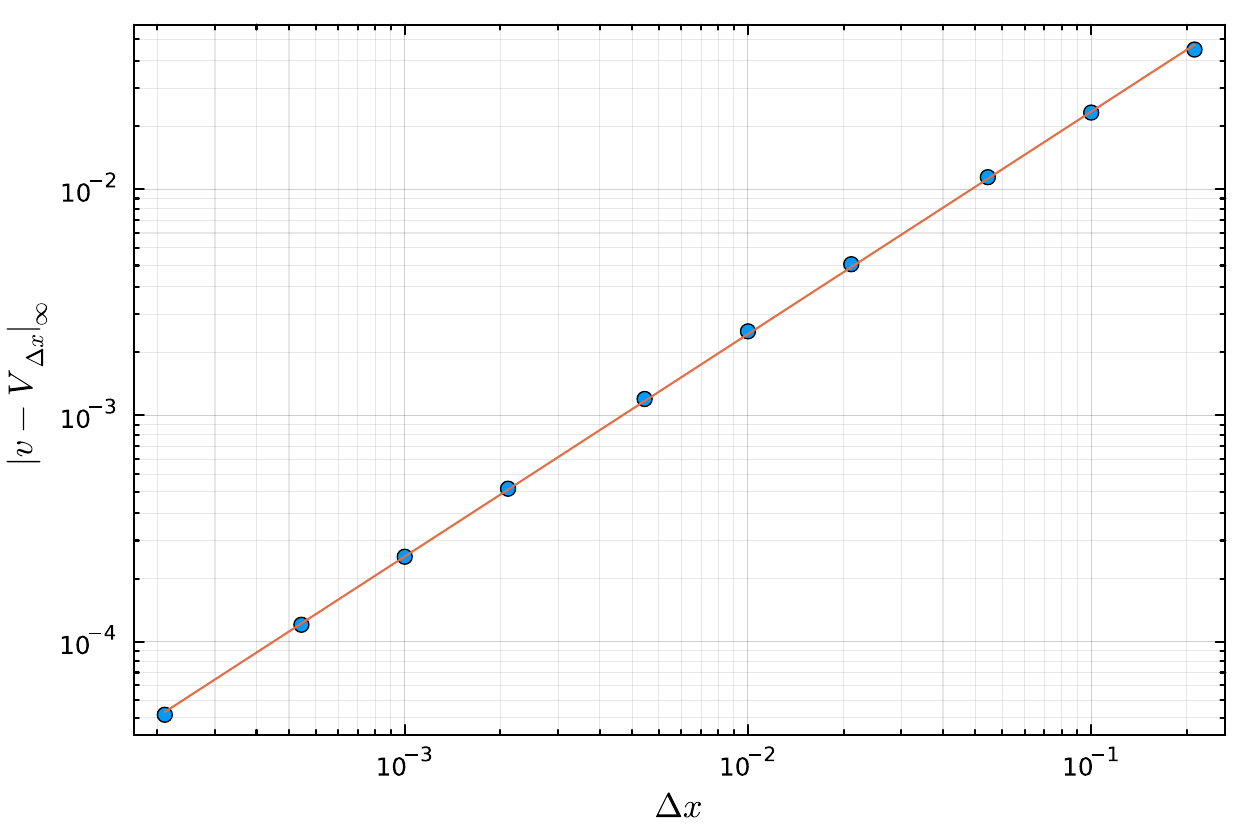}
  \caption{Error $\|v-V_{\Delta x}\|_{\infty}$ for various $\Delta x$.}
  \label{fig:HJB_err_u}
\end{figure}
\begin{figure}[th]
  \centering
  \includegraphics[width=0.8\textwidth]{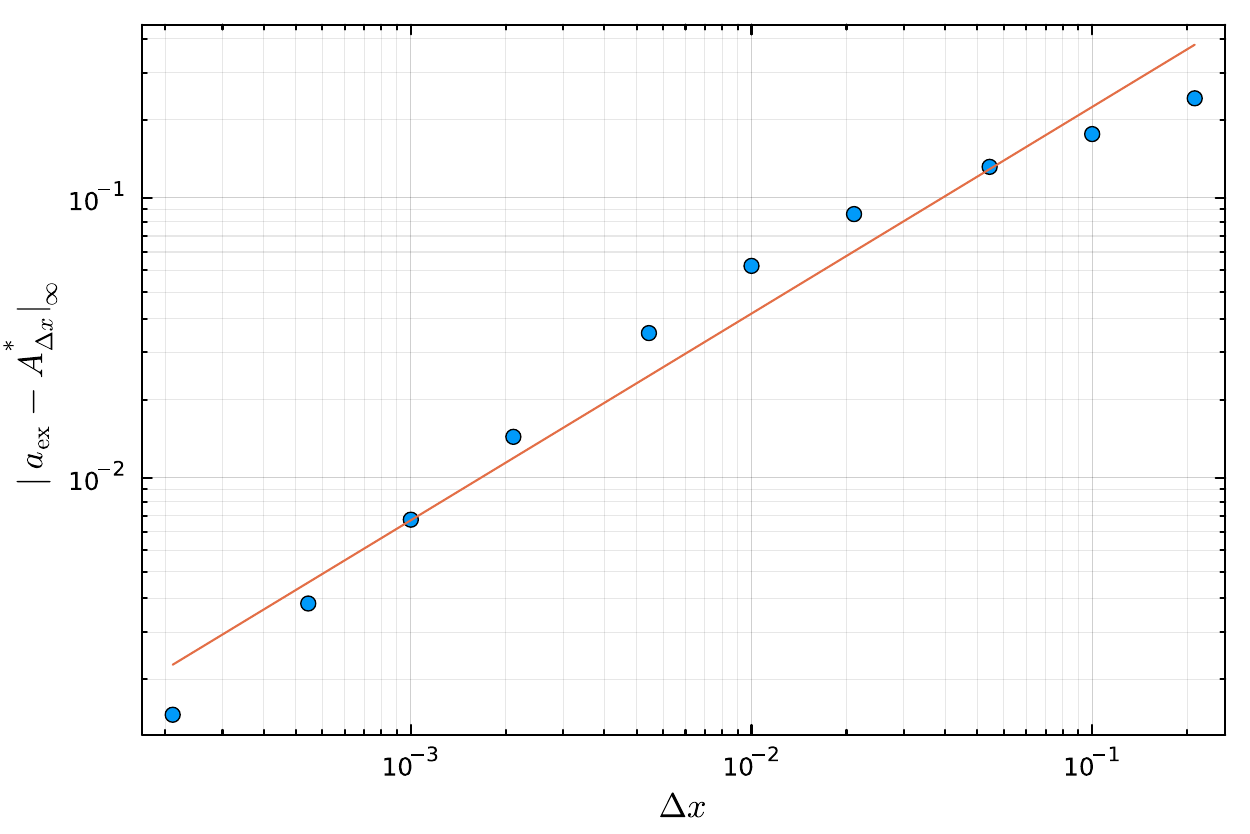}
  \caption{Error $\|a_{\mathrm{ex}}-A_{\Delta x}^*\|_{\infty}$ for various $\Delta x$.}
  \label{fig:HJB_err_a}
\end{figure}

Figure~\ref{fig:HJB_err_u} plots $\|v-V_{\Delta x}\|_{\infty}$ and Fig.~\ref{fig:HJB_err_a} plots $\|a_{\mathrm{ex}}-A_{\Delta x}^*\|_{\infty}$ for ten meshes ranging from $\Delta x=0.2$ to $\Delta x=2\times10^{-4}$.
The least-squares slope for the value function is $0.983$.
For the control the slope over the same ten meshes is $0.739$, and the rates between consecutive meshes rise towards $1$ as the mesh is refined, reaching $0.99$ on the two finest pairs.

The value-function rate agrees with the first-order rate predicted by Theorem~\ref{thm:convergence} for this example.
To interpret the control error, note that the exact control vanishes at the two points $x=0$ and $x=\pm1$, which the scheme treats differently.
The numerical value function is even about each of them, so $D^-V_{i,j}=-D^+V_{i,j}$ there.
At $x=0$ the value function has its minimum, so $D^+V_{i,j}>0$ and both candidates of \eqref{eq:smooth_discrete_objective} obtained from the one-sided difference quotients are cut off at $0$; the discrete minimizer is $0$, which is the exact value.
At $x=\pm1$, where the value function has its maximum, $D^+V_{i,j}<0$ and the two candidates are $\mp D^+V_{i,j}$, with the same objective value.
The discrete problem there has two minimizers even though the continuous one does not, which is the situation described in Remark~\ref{rem:discrete_nonunique}, and this is where the largest control error occurs.
Both candidates have magnitude proportional to $\Delta x$, so either admissible selection converges to the exact value $a_{\mathrm{ex}}(\pm1,t)=0$.
This is consistent with the robustness to arbitrary selections of the discrete minimizers expressed in Theorem~\ref{thm:convenience_input}, which is an almost-everywhere statement.

\subsection{Sparse-Control Example}

We next examine the sensitivity of discrete control selection to errors in the difference quotients when the continuous pointwise minimizer is nonunique.
Set $T=1$ and $E=[-1,1]$.
For $f(x,a)=a$, $g(x,a)=\lambda|a|$, and $v_T(x)=d_\Omega(x,0)^2$, the HJB equation is
\begin{align}\label{eq:l1_HJB}
  \begin{dcases}
    \partial_t  v(x,t) = -\min_{a\in E} \left[a\partial_x v(x,t) + \lambda|a| \right], & (x,t) \in \Omega\times(0,T), \\
    v(x,T) = v_T (x) = d_\Omega(x,0)^2,                                                & x\in \Omega,
  \end{dcases}
\end{align}
where $\lambda>0$ is a regularization parameter, which is set to $\lambda=0.1$ in these experiments.
On the representative interval $[-1,1)$ the terminal cost agrees with $x^2$; writing it through $d_\Omega$ makes it a function on $\Omega$.
This HJB equation corresponds to a sparse optimal control problem, which promotes sparsity of the control input in time~\citep{Ikeda2019Sparse}.
Its exact value function can be computed in closed form, while its pointwise minimizer is nonunique on a region of positive measure.

Since $\min_{a\in E}(ap+\lambda|a|)=-\max(0,|p|-\lambda)$, the periodic Hopf--Lax formula
\begin{align}\label{eq:l1_hopf_lax}
  v(x,t)=\min_{d_\Omega(x,y)\le\tau}\left\{d_\Omega(y,0)^2+\lambda d_\Omega(x,y)\right\},
  \qquad \tau\coloneqq T-t,
\end{align}
reduces to minimizing $(r-\delta)^2+\lambda\delta$ over $\delta\in[0,\min(r,\tau)]$, where $r\coloneqq d_\Omega(x,0)$ and $\delta$ is the distance travelled towards the origin.
The stationary point of that scalar problem is $r-\delta=\lambda/2$, and \eqref{eq:l1_hopf_lax} gives
\begin{align}\label{eq:l1_exact}
  v(x,t)=
  \begin{dcases}
    r^2,                    & r\le\lambda/2,                \\
    \lambda r-\lambda^2/4,  & \lambda/2<r\le\lambda/2+\tau, \\
    (r-\tau)^2+\lambda\tau, & r>\lambda/2+\tau.
  \end{dcases}
\end{align}
Away from the point $r=1$, which is represented by $x=\pm1$, the first region has $|\partial_x v|=2r\le\lambda$, which is why the terminal value is left unchanged there, while the third region has $|\partial_x v|>\lambda$, so that the unique minimizer is $a=-\mathrm{sign}(x)$.
On the intermediate region, however, $|\partial_x v|=\lambda$, and every $a\in[-1,0]$ for $x>0$ (resp. every $a\in[0,1]$ for $x<0$) attains the pointwise minimum.
The value function is thus uniquely determined, whereas the pointwise minimizer is not, and Assumption~\ref{assmp:control_uniqueness} is not satisfied in this example.
The terminal cost is irregular at $r=1$.
On the representative interval it agrees with $x^2$, and its one-sided derivatives there are $\partial_x v_T(1^-)=2$ and $\partial_x v_T((-1)^+)=-2$.
As a function on $\Omega$, it is therefore neither $\calC^1$ nor $\calC^2$ at that point.
Assumption~\ref{assmp:HJB} holds.
The set $E$ is compact, and $f$ and $g$ are continuous.
Since $f$ is independent of $x$, it satisfies the spatial Lipschitz condition; $v_T$ is continuous on $\Omega$.
The basic periodic setting is therefore in force, while the additional assumptions used for convergence of the difference quotients and controls are absent; this is the complement of Section~\ref{sec:numeric_1}, where all three assumptions hold.
Proposition~\ref{prop:value_uniform_convergence} accordingly applies to this example.
Proposition~\ref{prop:convergence_derivative} does not apply because Assumption~\ref{assmp:deriv} fails.
Theorem~\ref{thm:convenience_input} does not apply because it also requires Assumption~\ref{assmp:control_uniqueness}.
Theorem~\ref{thm:convergence} does not apply either, since \eqref{eq:l1_exact} is not a classical solution with H\"{o}lder continuous first derivatives.

The discrete minimization \eqref{eq:upwind_HJB_input} is evaluated directly from the upwind slopes, since its objective is piecewise linear in $a$ and therefore attains its minimum over $E$ at $-1$, $0$, or $1$.
The pointwise minimizer of the continuous problem is nonunique on the intermediate region, so the comparison for the control uses the reference control $A_{\Delta x}^{\mathrm{ref}}$ obtained by applying to the exact value function the same discrete minimization rule as the one used in the scheme \eqref{eq:upwind_HJB_input}, that is,
\begin{align}\label{eq:l1_reference_control}
  (A_{\Delta x}^{\mathrm{ref}})_{i,j} \in \argmin_{a\in E}\left(a^{+} D^+v(x_i,t_j)+a^{-} D^-v(x_i,t_j)+ \lambda|a|\right),
\end{align}
where $D^\pm$ denote the difference quotients of Section~\ref{sec:scheme} applied to the grid values of the exact solution.
The objective in \eqref{eq:l1_reference_control} is piecewise linear, so it attains its minimum over $E$ at $-1$, $0$, or $1$, and the selection rule is the one used in the scheme: whenever $a=0$ is among the minimizing candidates we select $a=0$, and any remaining tie between $a=-1$ and $a=1$ is resolved deterministically.
The second case arises at $r=1$, where symmetry makes the two candidates $a=\pm1$ attain the same negative value.

On a stencil contained in the linear intermediate region, the relevant one-sided difference quotient of \eqref{eq:l1_exact} has magnitude $\lambda$, with the sign determined by the side of the origin, and the tie in \eqref{eq:l1_reference_control} is attained exactly.
In floating-point arithmetic, we detect this tie using the tolerance $64\,\varepsilon_{\mathrm{mach}}/\Delta x$, where $\varepsilon_{\mathrm{mach}}$ is the machine epsilon.
The discrepancies between the difference quotients are of order $\Delta x$ and exceed this tolerance on every mesh used here, and the resulting selection agrees with exact arithmetic at every grid point.

\begin{figure}[th]
  \centering
  \includegraphics[width=0.8\textwidth]{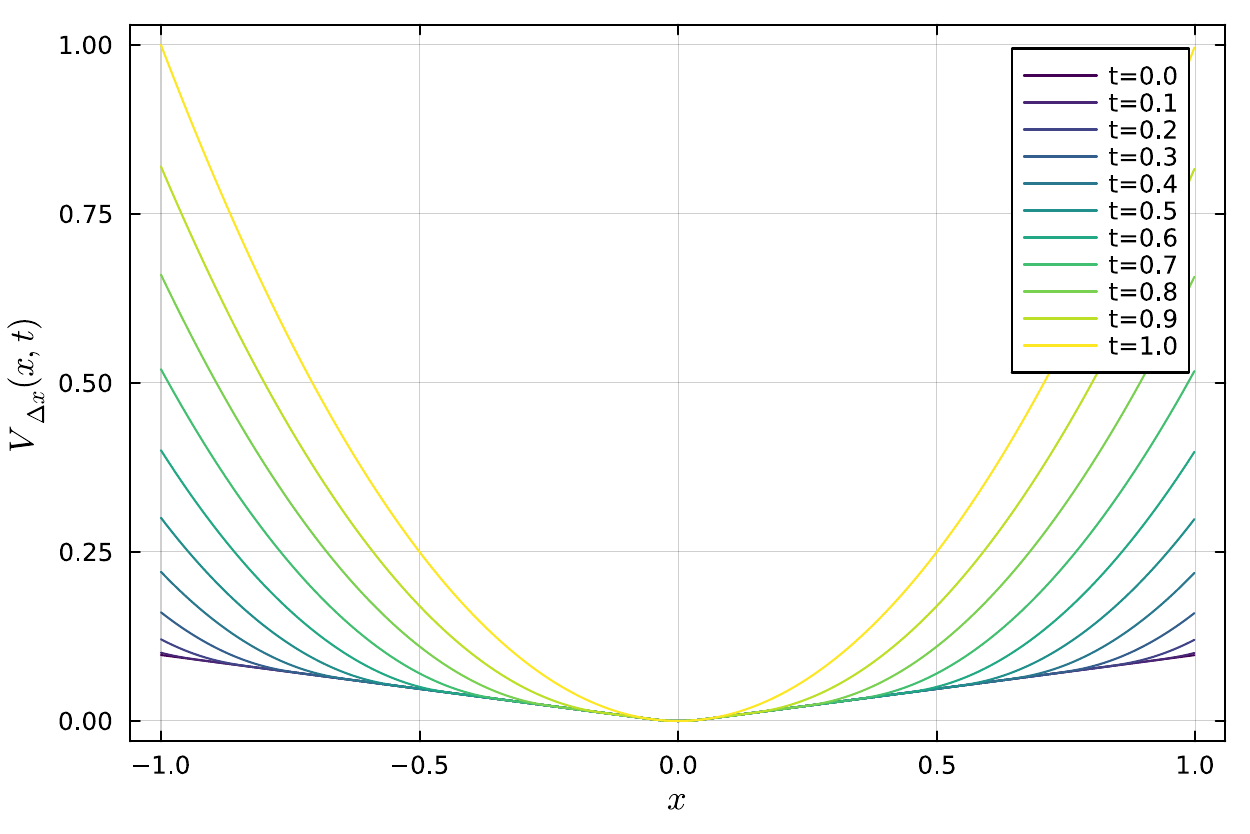}
  \caption{Time response of the numerical value function $V_{\Delta x}$.}
  \label{fig:HJB12_v}
\end{figure}
\begin{figure}[th]
  \centering
  \includegraphics[width=0.8\textwidth]{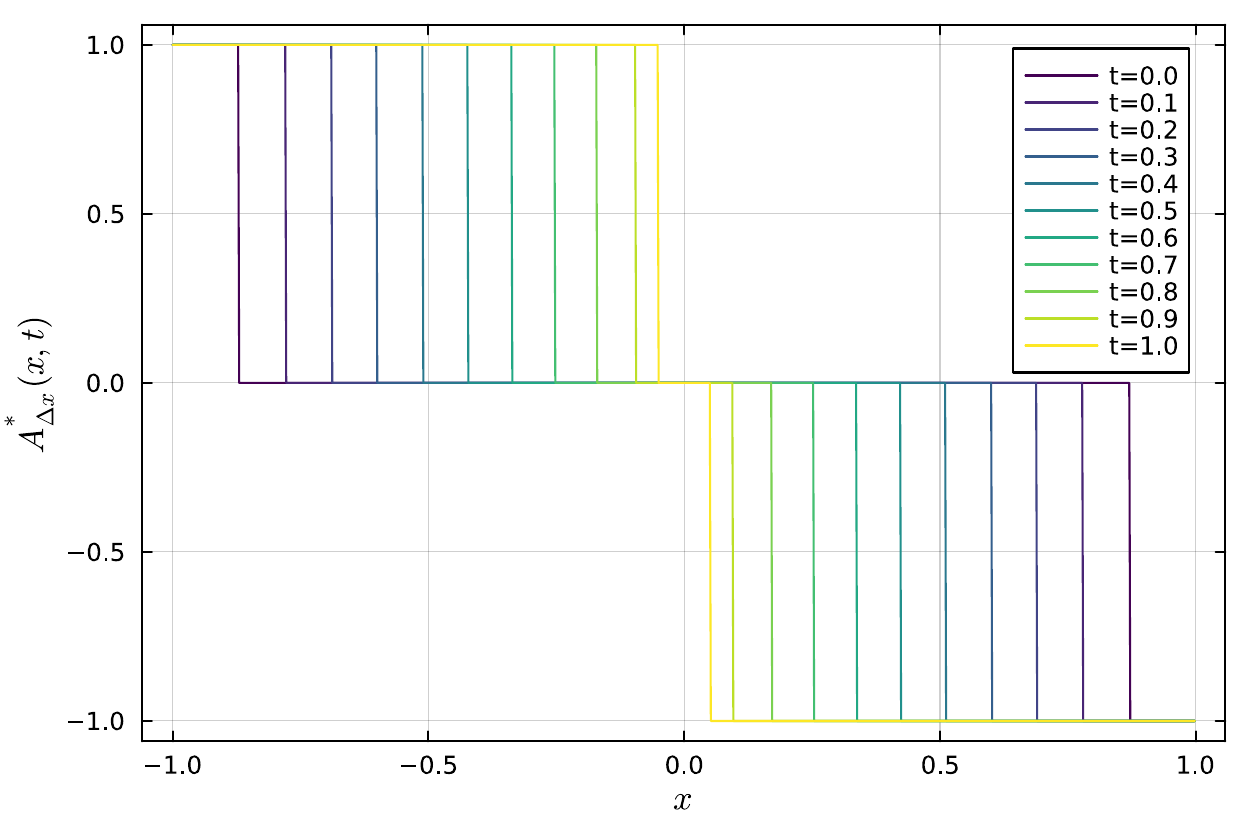}
  \caption{Time response of the numerical control $A_{\Delta x}^*$.}
  \label{fig:HJB12_a}
\end{figure}

Figures~\ref{fig:HJB12_v} and \ref{fig:HJB12_a} show the time responses of $V_{\Delta x}$ and $A_{\Delta x}^*$ for $\Delta x=0.002$.
The discrete control is zero in a central region containing the target state $x=0$, while it takes the maximum or minimum admissible value in the outer regions, and the switching boundary moves as the terminal time is approached.
In particular, the zero control remains optimal in a neighborhood of the target.

\begin{figure}[th]
  \centering
  \includegraphics[width=0.8\textwidth]{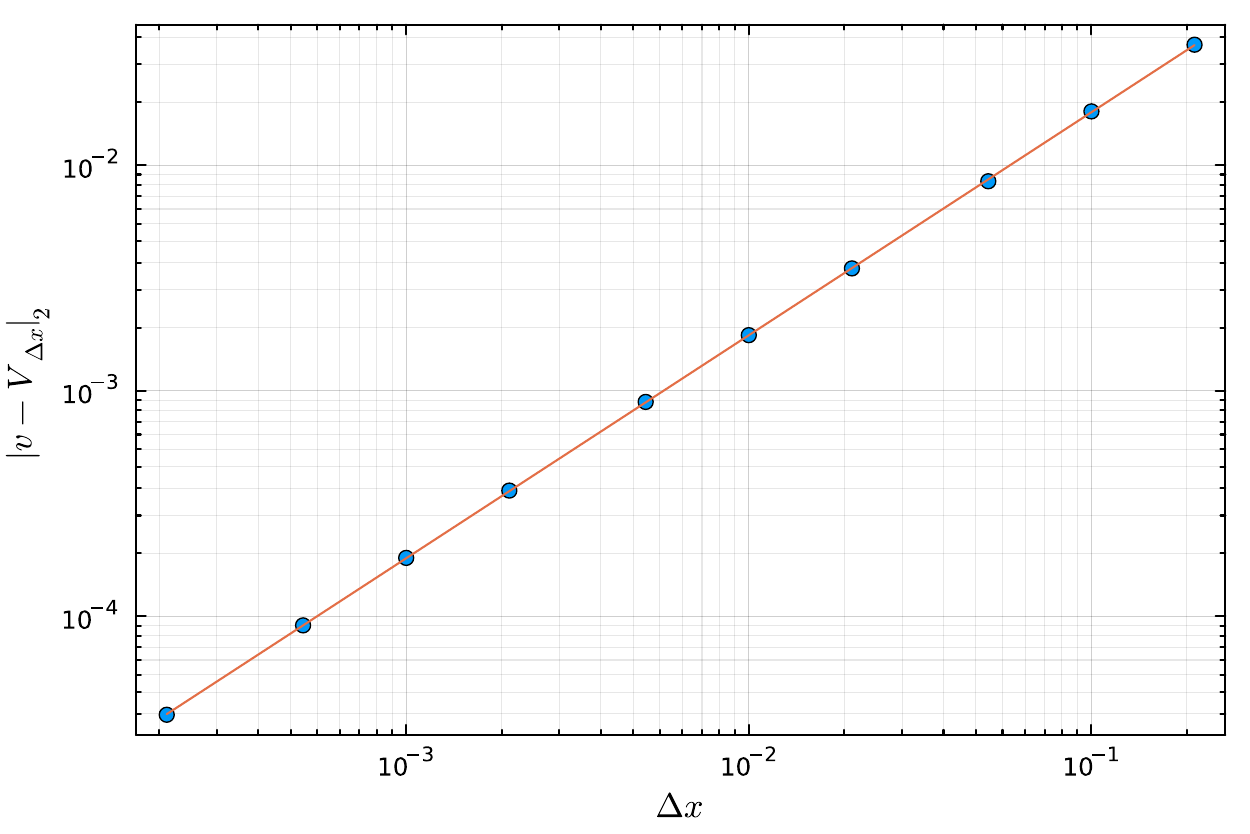}
  \caption{Error $\|v-V_{\Delta x}\|_{2}$ for various $\Delta x$.}
  \label{fig:HJB12_err_u}
\end{figure}
\begin{figure}[th]
  \centering
  \includegraphics[width=0.8\textwidth]{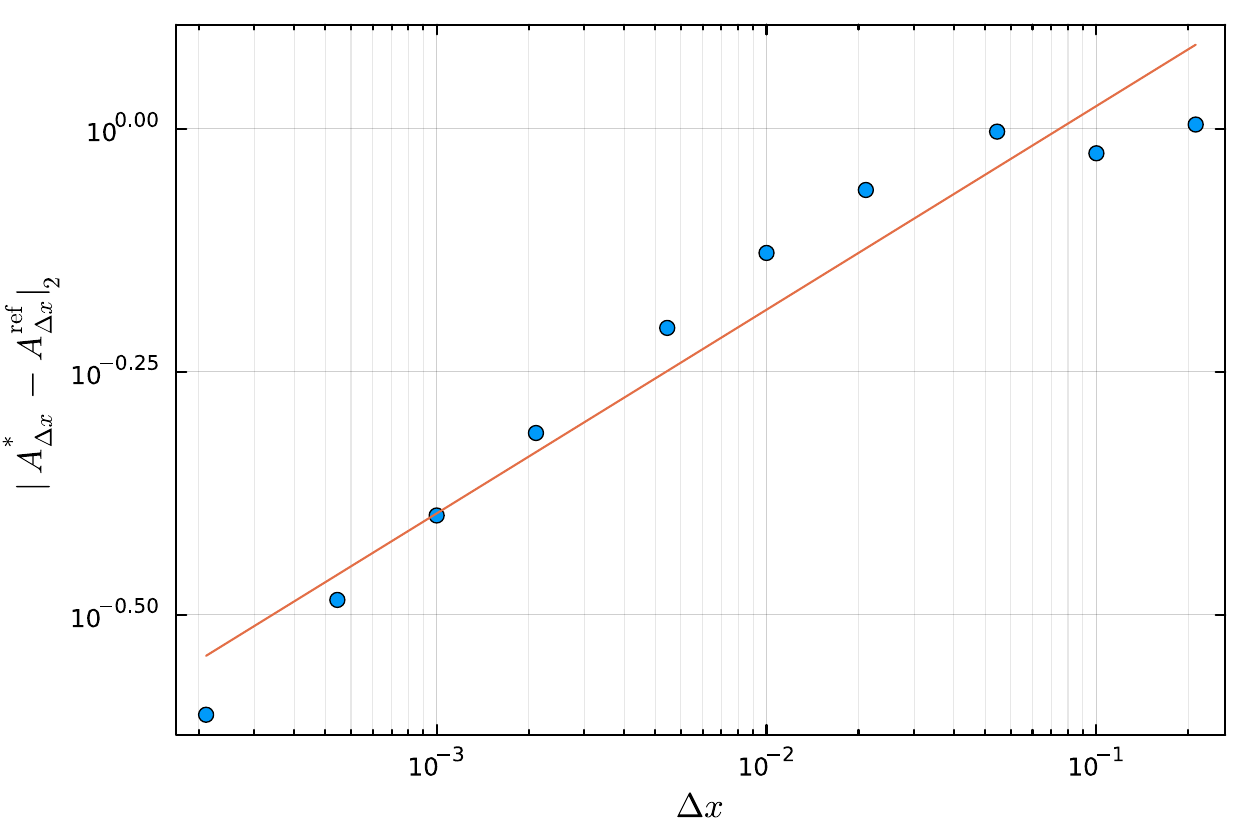}
  \caption{Control discrepancy $\|A_{\Delta x}^*-A_{\Delta x}^{\mathrm{ref}}\|_{2}$ for various $\Delta x$.}
  \label{fig:HJB12_err_a}
\end{figure}

Figure~\ref{fig:HJB12_err_u} shows $\|v-V_{\Delta x}\|_{2}$, the error of the numerical value function with respect to \eqref{eq:l1_exact}.
Figure~\ref{fig:HJB12_err_a} shows $\|A_{\Delta x}^*-A_{\Delta x}^{\mathrm{ref}}\|_{2}$, the discrepancy between two discrete controls selected by the same rule from different value functions.
Both errors decrease under mesh refinement.
The least-squares slope is $0.989$ for the value function, compared with $0.983$ in the smooth example of Section~\ref{sec:numeric_1}, and $0.210$ for the control discrepancy.

This behavior reflects the structure described above: on the intermediate region $|\partial_x v|=\lambda$, the value at which the selected control switches, so an arbitrarily small perturbation of the numerical difference quotient changes it by an $O(1)$ amount.

\section{Conclusion}\label{sec:conclusion}

In this paper, we established convergence properties of an upwind finite-difference scheme for HJB equations at the levels of the value function, its one-sided numerical difference quotients, and the feedback control selected by the discrete minimization.
Suppose that the value function $v$ is a classical solution, with $\partial_t v$ H\"{o}lder continuous in time and $\partial_x v$ H\"{o}lder continuous in space for a common exponent $\eta\in(0,1]$.
Then the grid-point error is of order $O(\Delta t^\eta+\Delta x^\eta)$.
Under proportional refinement of the temporal and spatial meshes, this yields first-order convergence when the first derivatives are Lipschitz continuous.
Under additional spatial regularity of the problem data, we further proved strong $L^1$ convergence of both one-sided numerical difference quotients to the spatial derivative of the value function.
If the continuous pointwise minimizer is unique, every selection of discrete minimizers converges almost everywhere, along a common subsequence, to the continuous feedback control.
Discrete semiconcavity and periodicity bound the numerical slopes and their spatial total variation; discrete spatial averaging combines these bounds with value convergence to prove strong convergence of the difference quotients.
The numerical experiments support the predicted convergence behavior for a smooth periodic problem and illustrate the sensitivity of control selection in a sparse-control problem where the additional regularity and minimizer-uniqueness assumptions fail.
The spatial variation argument for convergence of the difference quotients and controls is one-dimensional, and extending these results to higher dimensions requires suitable estimates for multidimensional numerical gradients.

\backmatter

\bmhead{Acknowledgements}

Generative-AI tools were used to assist with manuscript proofreading.

\section*{Declarations}

\begin{itemize}
  \item \textbf{Funding.} This work was supported by JSPS KAKENHI Grant Numbers JP23K03175 (Scientific Research (C)) and JP26K22267 (Challenging Research (Exploratory)).
  \item \textbf{Conflict of interest.} The authors have no competing interests to declare that are relevant to the content of this article.
  \item \textbf{Data availability.} The data that support the findings of this study are available from the corresponding author upon reasonable request.
  \item \textbf{Author contributions.} Conceptualization: D.I.; Methodology: D.I.; Formal analysis and investigation: D.I. and Y.I.; Software: D.I.; Validation: D.I., Y.I., T.K., and N.H.; Writing - original draft preparation, review, and editing: D.I.; Supervision: N.S. and H.Y.
\end{itemize}

\begin{appendices}

  \section{Standard Results Used in the Analysis}\label{app:standard}

  We recall two existing results: well-posedness of the finite-horizon HJB equation and convergence of monotone approximation schemes.

  \subsection{Well-posedness of the HJB equation}

  The following is the well-posedness result behind Proposition~\ref{prop:periodic_HJB}.
  It is the specialization of \cite[Chapter~II, Theorem~10.1]{Fleming2006Controlled} to the deterministic finite-horizon problem considered here.

  \begin{theorem}[Well-posedness of the finite-horizon HJB equation]\label{thm:FS-wellposed}
    Consider the optimal control problem \eqref{eq:ODE_system}--\eqref{eq:value_function} and suppose that
    \begin{enumerate}
      \item the control set $E\subset\bbR^m$ is nonempty and compact;
      \item $f$ is bounded and continuous on $\bbR^n\times E$, and there exists $L_f>0$ such that
            \begin{align}
              |f(x,a)-f(y,a)|\le L_f|x-y|,
            \end{align}
            for every $x,y\in\bbR^n$ and $a\in E$;
      \item $g$ is bounded and uniformly continuous on $\bbR^n\times E$;
      \item $v_T$ is bounded and uniformly continuous on $\bbR^n$.
    \end{enumerate}
    Then the value function $v$ defined in \eqref{eq:value_function} is bounded and uniformly continuous on $\bbR^n\times[0,T]$, and it is the unique viscosity solution of \eqref{eq:HJB} in the class of bounded uniformly continuous functions on $\bbR^n\times[0,T]$.
  \end{theorem}

  \subsection{Convergence of monotone schemes}\label{app:BS-theorem}

  We recall the first-order case of \cite[Theorem~2.1]{Barles1991Convergence}, with the boundary formulation of its Definition~1.1.
  Let $D\subset\bbR^d$ be open, and let $B(\overline D)$ be the space of bounded real-valued functions on $\overline D$, equipped with the pointwise supremum norm.
  A locally bounded operator $F(y,r,p)$, defined on $\overline D\times\bbR\times\bbR^d$, encodes the equation and boundary data.
  Write $F_*,F^*$ for its lower and upper semicontinuous envelopes in all arguments, with the point argument restricted to $\overline D$.
  An upper-semicontinuous function $u$ is a viscosity subsolution of $F(y,u,Du)=0$ if every $\calC^2$ upper touching test function $\phi$ satisfies $F_*(y,u(y),D\phi(y))\le0$.
  For lower-semicontinuous supersolutions the corresponding inequality is $F^*(y,u(y),D\phi(y))\ge0$ at lower contacts.
  All contacts are relative to $\overline D$, including at boundary points.

  \begin{theorem}[Barles--Souganidis convergence theorem]\label{thm:BS-convergence}
    Let $\widehat F:[0,\infty)\times\overline D\times\bbR\times B(\overline D)\to\bbR$ be locally bounded.
    Suppose that the following conditions hold:
    \begin{enumerate}
      \item \emph{Monotonicity} \cite[(2.2)]{Barles1991Convergence}: for every $\rho\ge0$, $y\in\overline D$, $r\in\bbR$, and $U,Z\in B(\overline D)$,
            \begin{align*}
              U\ge Z\quad\Longrightarrow\quad \widehat F(\rho,y,r,U)\le \widehat F(\rho,y,r,Z).
            \end{align*}
      \item \emph{Existence and stability} \cite[(2.3)]{Barles1991Convergence}: for every $\rho>0$, there is $u^\rho\in B(\overline D)$ satisfying
            \begin{align*}
              \widehat F(\rho,y,u^\rho(y),u^\rho)=0\qquad(y\in\overline D),
            \end{align*}
            with $\sup_{\rho>0}\|u^\rho\|_\infty<\infty$.
      \item \emph{Consistency} \cite[(2.4a,b)]{Barles1991Convergence}: for every $y\in\overline D$ and every bounded smooth test function $\phi\in\calC^\infty(\overline D)$,
            \begin{align*}
              \limsup_{\substack{\rho\downarrow0,\ z\to y\\\xi\to0}}
              \frac{\widehat F(\rho,z,\phi(z)+\xi,\phi+\xi)}{\rho}
               & \le F^*(y,\phi(y),D\phi(y)), \\
              \liminf_{\substack{\rho\downarrow0,\ z\to y\\\xi\to0}}
              \frac{\widehat F(\rho,z,\phi(z)+\xi,\phi+\xi)}{\rho}
               & \ge F_*(y,\phi(y),D\phi(y)),
            \end{align*}
            where $z\in\overline D$ and $\xi\in\bbR$.
      \item \emph{Comparison} \cite[(2.5)]{Barles1991Convergence}: every bounded upper-semicontinuous viscosity subsolution $u$ and bounded lower-semicontinuous viscosity supersolution $z$ satisfy $u\le z$ on $\overline D$.
    \end{enumerate}
    Then $u^\rho$ converges locally uniformly on $\overline D$, as $\rho\downarrow0$, to the unique bounded continuous viscosity solution of $F(y,u,Du)=0$.
  \end{theorem}

  \section{Auxiliary Verifications and Proofs}\label{app:proofs}

  \subsection{Proof of Proposition~\ref{prop:value_uniform_convergence}}\label{app:BS-scheme}

  \begin{proof}
    We apply Theorem~\ref{thm:BS-convergence} using the standard upwind properties; see also \cite[Chapter~IX, Sections~3--4]{Fleming2006Controlled} for viscosity-solution convergence analysis of such schemes.
    We work on the tail where \eqref{eq:value-convergence-CFL} holds, using the periodic lifts, projection $\pi_k$, and fixed-mesh notation of Section~\ref{sec:convergence}.
    Norms in this proof are pointwise supremum norms; in particular, $\|g\|_\infty=\max_{\Omega\times E}|g|$.

    \emph{Boundary formulation.}
    As in the proof of Theorem~\ref{thm:convergence}, we first address the terminal time before estimating the scheme.
    Here we encode the terminal condition in the viscosity boundary formulation, rather than extend a classical HJB identity to $t=T$.
    For $D=\bbR\times(0,T)$, we use the operator
    \begin{align}\label{eq:BS-boundary-operator}
      F(x,t,r,p,b)=
      \begin{cases}
        -b-H(x,p), & 0\le t<T, \\
        r-v_T(x),  & t=T.
      \end{cases}
    \end{align}
    Here $p,b$ are the spatial and temporal derivative arguments, and tests are relative to $\overline D$, including PDE tests at $t=0$.

    \emph{Upwind properties.}
    Under CFL, each control gives nonnegative update coefficients with sum one, so the minimum over control-dependent updates is monotone and commutes with addition of constants.
    Backward recursion gives existence and stability:
    \begin{align}\label{eq:stability-sup}
      \max_{i,j}|V_{i,j}|\le\|v_T\|_\infty+T\|g\|_\infty.
    \end{align}
    The first backward step has total neighboring weight at most one, hence
    \begin{align}\label{eq:value-one-step}
      \max_i|V_{i,N_t-1}-v_T(x_i)|
      \le\omega_{v_T}(\Delta x_k)+\Delta t_k\|g\|_\infty\longrightarrow0,
    \end{align}
    where $\omega_{v_T}(r)$ bounds changes of $v_T$ over periodic distances at most $r$ and tends to zero with $r$.

    \emph{Scheme representation and consistency.}
    Choose $\rho_k\downarrow0$ strictly with $0<\rho_k\le\Delta t_k$.
    For $\rho_{k+1}<\rho\le\rho_k$, set $u^\rho=V_k$ in Theorem~\ref{thm:BS-convergence} and, on the cell $\pi_k(x)=x_i$, define
    \begin{align}\label{eq:BS-residual}
      \frac{\widehat F(\rho,x,t,r,U)}{\rho}
      =
      \begin{cases}
        \begin{aligned}
           & \displaystyle\frac{r-U(x,t+\Delta t_k)}{\Delta t_k}                          \\
           & {}-\widehat H\bigl(\pi_k(x),D^-U(x,t+\Delta t_k),D^+U(x,t+\Delta t_k)\bigr),
        \end{aligned}
                       & 0\le t\le T-\Delta t_k, \\[1mm]
        r-V_{i,N_t-1}, & T-\Delta t_k<t<T,       \\
        r-v_T(x_i),    & t=T.
      \end{cases}
    \end{align}
    We use the first mesh for $\rho>\rho_1$ and set $\widehat F(0,\cdot)=0$.
    The projection is read through the periodic equivalence class of $x$; the differences use $x,x\pm\Delta x_k$.
    The value $V_{i,N_t-1}$ is the known first backward update of $v_T$, independent of $U$.
    Thus $\widehat F(\rho,x,t,V_k(x,t),V_k)=0$ everywhere, including cell boundaries and both time endpoints.
    Monotonicity gives the required nonincrease in $U$, while $\rho/\Delta t_k\le1$ near zero and \eqref{eq:stability-sup} ensure local boundedness of $\widehat F$.

    For a bounded smooth $\phi$ and constants $\xi\to0$, the perturbations cancel on the update branch, and the projection bound and difference-quotient consistency with \eqref{eq:Hamiltonian-diagonal} give $\widehat F(\rho,x,t,\phi(x,t)+\xi,\phi+\xi)/\rho\to-\partial_t\phi-H(x,\partial_x\phi)$, also at $t=0$.
    At $t=T$, the other two branches converge to $\phi(x,T)-v_T(x)$ by \eqref{eq:value-one-step}.
    These are precisely the limits required by $F_*=\min\{-b-H,r-v_T\}$ and $F^*=\max\{-b-H,r-v_T\}$ at $T$; away from $T$, both envelopes equal $-b-H$.

    \emph{Comparison on the closed strip.}
    Let $u,z$ be any bounded upper-semicontinuous subsolution and lower-semicontinuous supersolution of \eqref{eq:BS-boundary-operator}, not necessarily periodic.
    The data satisfy
    \begin{align}\label{eq:comparison-H-continuity}
      |H(x,p)-H(y,p)| & \le L_f|p||x-y|+\omega_g(|x-y|), \\
      |H(x,p)-H(x,q)| & \le F_0|p-q|,
    \end{align}
    where $\omega_g(r)$ bounds spatial changes of $g$ over distances at most $r$, uniformly over $E$, and tends to zero with $r$.
    To localize the extrema, set $\chi(x)=\sqrt{1+x^2}$, so that
    $\chi(x)\to\infty$ as $|x|\to\infty$ and $|\chi'(x)|\le1$.
    Choose a smooth periodic $\psi$ with $\|\psi-v_T\|_\infty\le\varepsilon$ and put $C_\psi=F_0\|\psi'\|_\infty+\|g\|_\infty+1$.
    For $\delta>0$ with $F_0\delta<1$, the maximum argument for $u-\psi-C_\psi(T-t)-\delta\chi$ and the corresponding minimum argument for $z-\psi+C_\psi(T-t)+\delta\chi$ give $u(x,T)\le\psi(x)+\varepsilon+\delta\chi(x)$ and $z(x,T)\ge\psi(x)-\varepsilon-\delta\chi(x)$.
    Indeed, an extremum violating these bounds contradicts the strict PDE test inequality, or the relaxed terminal condition at $T$; contacts at $0$ are allowed.
    Sending $\delta\downarrow0$ and then $\varepsilon\downarrow0$ yields $u(x,T)\le v_T(x)\le z(x,T)$.

    If $u-z$ is positive somewhere, choose $\eta,\delta>0$ with $2F_0\delta<\eta$ such that
    \begin{align}\label{eq:comparison-penalty}
      \begin{split}
        u(x,t)-z(y,s)-\eta(T-t)-\delta\{\chi(x)+\chi(y)\}
        -\frac\lambda2\bigl(|x-y|^2+|t-s|^2\bigr)
      \end{split}
    \end{align}
    has a positive diagonal value.
    Its maximizers lie in a fixed compact set, and \cite[Lemma~3.1]{Crandall1992UsersGuide} gives $\lambda(|x-y|^2+|t-s|^2)\to0$ at these points as $\lambda\to\infty$.
    The terminal inequalities exclude limiting contact times at $T$; at $0$ the relative PDE tests apply.
    Subtracting the two test inequalities at a maximizer and using \eqref{eq:comparison-H-continuity} gives
    \begin{align}
      \eta\le L_f\lambda|x-y|^2+\omega_g(|x-y|)+2F_0\delta
      \longrightarrow2F_0\delta<\eta,
    \end{align}
    a contradiction.

    Theorem~\ref{thm:BS-convergence} therefore gives local uniform convergence of $V_k$ on $\overline D$ to a continuous periodic viscosity solution.
    Its periodicity makes it bounded and uniformly continuous, so Proposition~\ref{prop:periodic_HJB} identifies it with $v$.
    Restricting to one closed spatial period gives \eqref{eq:value-uniform-convergence}.
  \end{proof}

  \subsection{Proof of Lemma~\ref{lmm:discrete_semiconcavity}}\label{app:semiconcavity}

  \begin{proof}
    We derive a scalar recurrence for the largest normalized second difference and then iterate it to obtain the uniform bound.

    \emph{Step 1: The scalar recurrence to be proved.}
    The constants in the statement are finite under Assumption~\ref{assmp:deriv}.
    For $j=0,\ldots,N_t$, introduce the normalized second differences
    \begin{align}\label{eq:z-def}
      z_{i,j}
      \coloneqq
      \frac{V_{i+1,j}-2V_{i,j}+V_{i-1,j}}{(\Delta x)^2}
      =\frac{p_{i,j}-p_{i-1,j}}{\Delta x},
      \qquad
      K_j\coloneqq\max_i z_{i,j}.
    \end{align}
    Summing over one period, both sums telescope and therefore vanish:
    \begin{align}\label{eq:telescope}
      \sum_{i=0}^{N_x-1}p_{i,j}=0,
      \qquad
      \sum_{i=0}^{N_x-1}z_{i,j}=0.
    \end{align}
    Hence $K_j\ge0$.
    For $j=1,\ldots,N_t$, we will prove the scalar recurrence
    \begin{align}\label{eq:K_recurrence}
      K_{j-1}\le\left[1+(3F_1+2F_2)\Delta t\right]K_j+G_2\Delta t.
    \end{align}
    Once \eqref{eq:K_recurrence} is established, its iteration from the terminal data will give $K_j\le C_{\mathrm{sc}}$.

    \emph{Step 2: A one-step estimate for the second differences.}
    We next derive a one-step upper estimate for the second differences by fixing a minimizing control at the central grid point and reusing it at the two neighboring points.
    Fix $i$ and $j\in\{1,\ldots,N_t\}$.
    By compactness of $E$ and continuity of the discrete objective, choose an arbitrary minimizer
    \begin{align}
      a_{i,j}\in\argmin_{a\in E}
      \left\{
      f^-(x_i,a)p_{i-1,j}
      +f^+(x_i,a)p_{i,j}
      +g(x_i,a)
      \right\}.
    \end{align}
    We now evaluate the objective defining $\widehat H$ at the three grid points $x_{i-1}$, $x_i$, and $x_{i+1}$ using this one control $a_{i,j}$.
    Accordingly, set
    \begin{align}
      f_\ell\coloneqq f(x_\ell,a_{i,j}),
      \qquad
      g_\ell\coloneqq g(x_\ell,a_{i,j}),
    \end{align}
    and
    \begin{align}
      q_\ell\coloneqq f_\ell^+p_{\ell,j}+f_\ell^-p_{\ell-1,j}+g_\ell,
    \end{align}
    so that $q_\ell$ is the value at $a=a_{i,j}$ of the objective whose minimum is $\widehat H(x_\ell,p_{\ell-1,j},p_{\ell,j})$.
    At $x_i$ this value is that minimum, by the choice of $a_{i,j}$.
    At $x_{i\pm1}$ the control $a_{i,j}$ is only one admissible choice among others, so the minimum there cannot exceed it.
    Hence
    \begin{align}\label{eq:frozen_control_comparison}
      \begin{split}
        \widehat H(x_i,p_{i-1,j},p_{i,j})         & =q_i,        \\
        \widehat H(x_{i+1},p_{i,j},p_{i+1,j})     & \le q_{i+1}, \\
        \widehat H(x_{i-1},p_{i-2,j},p_{i-1,j}) & \le q_{i-1}.
      \end{split}
    \end{align}
    The two neighboring values enter the second difference with positive sign, so the relations in \eqref{eq:frozen_control_comparison} bound it from above.
    Using a single control at all three points expresses $f_{i+1}-2f_i+f_{i-1}$ as the second difference of the function $x\mapsto f(x,a_{i,j})$, which Assumption~\ref{assmp:deriv} bounds, whereas minimizers chosen separately at each spatial grid point need not depend regularly on its index.
    Taking the second difference of the backward update \eqref{eq:backward-update} therefore yields
    \begin{align}\label{eq:second-difference-step}
      (\Delta x)^2z_{i,j-1}
      \le (\Delta x)^2z_{i,j}
      +\Delta t\left(q_{i+1}-2q_i+q_{i-1}\right).
    \end{align}

    We rewrite this estimate as a recurrence involving three neighboring second differences and a remainder.
    Expanding the second difference gives
    \begin{align}\label{eq:objective-second-difference}
      \begin{split}
        q_{i+1}-2q_i+q_{i-1}
         & =\left(f_{i+1}-2f_i+f_{i-1}\right)p_{i,j}         \\
         & \quad+f_{i+1}^+\left(p_{i+1,j}-p_{i,j}\right)     \\
         & \quad+\left(2f_i^- -f_{i-1}\right)
        \left(p_{i,j}-p_{i-1,j}\right)                       \\
         & \quad-f_{i-1}^-\left(p_{i-1,j}-p_{i-2,j}\right) \\
         & \quad+g_{i+1}-2g_i+g_{i-1}.
      \end{split}
    \end{align}
    Thus, no second difference of the positive or negative parts of the sampled dynamics is estimated separately, which is essential because $x\mapsto f^\pm(x,a_{i,j})$ need not be of class $\calC^2$ even when $x\mapsto f(x,a_{i,j})$ is.
    Substituting \eqref{eq:objective-second-difference} into \eqref{eq:second-difference-step}, dividing by $(\Delta x)^2$, and using the definition \eqref{eq:z-def} of $z_{\ell,j}$, we obtain
    \begin{align}\label{eq:z-recurrence}
      z_{i,j-1}
      \le \gamma_i^0 z_{i,j}+\gamma_i^+ z_{i+1,j}+\gamma_i^- z_{i-1,j}+\Delta t R_i,
    \end{align}
    where
    \begin{align}
      \begin{split}
        \gamma_i^0 & \coloneqq1+\frac{\Delta t}{\Delta x}\left(2f_i^- -f_{i-1}\right), \\
        \gamma_i^+ & \coloneqq\frac{\Delta t}{\Delta x}f_{i+1}^+,                      \\
        \gamma_i^- & \coloneqq-\frac{\Delta t}{\Delta x}f_{i-1}^-,
      \end{split}
    \end{align}
    and
    \begin{align}
      R_i\coloneqq
      \frac{
        (f_{i+1}-2f_i+f_{i-1})p_{i,j}
        +g_{i+1}-2g_i+g_{i-1}
      }{(\Delta x)^2}.
    \end{align}

    \emph{Step 3: Bounds for the weights and the remainder.}
    To obtain \eqref{eq:K_recurrence} from \eqref{eq:z-recurrence}, we will establish three bounds.
    We need nonnegative weights to replace the three second differences by their common upper bound $K_j$:
    \begin{align}\label{eq:gamma-nonneg}
      \gamma_i^0\ge0,
      \qquad \gamma_i^+\ge0,
      \qquad \gamma_i^-\ge0.
    \end{align}
    We also need to bound their total contribution to the growth of $K_j$:
    \begin{align}\label{eq:gamma-sum}
      \gamma_i^0+\gamma_i^++\gamma_i^-\le1+3F_1\Delta t.
    \end{align}
    Finally, we need to control the remainder in terms of $K_j$:
    \begin{align}\label{eq:R-bound}
      R_i\le2F_2K_j+G_2.
    \end{align}

    We first verify that all three weights in \eqref{eq:z-recurrence} are nonnegative.
    The definitions give $\gamma_i^+\ge0$ and $\gamma_i^-\ge0$.
    To bound $\gamma_i^0$ from below, we use $2f_i^-=f_i-|f_i|$ to write
    \begin{align}
      2f_i^- -f_{i-1}=-|f_i|+(f_i-f_{i-1}).
    \end{align}
    Since $|f_i|\le F_0$ and $|f_i-f_{i-1}|\le F_1\Delta x$, the mesh condition \eqref{eq:mesh-CFL} of Lemma~\ref{lmm:discrete_semiconcavity} gives
    \begin{align}
      \gamma_i^0
      \ge1-\frac{\Delta t}{\Delta x}F_0-\Delta t F_1
      \ge0.
    \end{align}
    This proves \eqref{eq:gamma-nonneg}.

    We next obtain an upper bound for the sum of the three weights.
    Their sum satisfies
    \begin{align}
      \gamma_i^0+\gamma_i^++\gamma_i^--1
      =\frac{\Delta t}{\Delta x}
      \left[
        (f_i-f_{i-1})
        +(f_{i+1}^+-f_i^+)
        +(f_i^- -f_{i-1}^-)
        \right].
    \end{align}
    The maps $r\mapsto r^+$ and $r\mapsto r^-$ are $1$-Lipschitz. Hence
    \begin{align}
      \begin{split}
        |f_i-f_{i-1}|      & \le F_1\Delta x, \\
        |f_{i+1}^+-f_i^+|  & \le F_1\Delta x, \\
        |f_i^- -f_{i-1}^-| & \le F_1\Delta x,
      \end{split}
    \end{align}
    and consequently \eqref{eq:gamma-sum} holds.

    We now estimate the remainder $R_i$ by using the regularity of the coefficients at the fixed control $a_{i,j}$.
    To control the factor $p_{i,j}$ in $R_i$, we repeat the periodic telescoping argument used to derive \eqref{eq:p-bound-T} in Lemma~\ref{lmm:slope_variation}, with $K_j$ in place of $C_{\mathrm{sc}}$.
    The definition \eqref{eq:z-def} of $K_j$ gives
    \begin{align}\label{eq:slope-step}
      p_{i,j}-p_{i-1,j}\le \Delta x K_j
    \end{align}
    for every $i$.
    Summing these increments along the positive periodic path from a minimizing index to a maximizing index and using $N_x\Delta x=2$ yields
    \begin{align}
      \max_i p_{i,j}-\min_i p_{i,j}\le2K_j.
    \end{align}
    The zero-mean identity $\sum_i p_{i,j}=0$ in \eqref{eq:telescope} then implies
    \begin{align}\label{eq:p-bound}
      |p_{i,j}|\le2K_j
    \end{align}
    for every $i$.
    Since the same control $a_{i,j}$ is used at all three points, the values $f_\ell$ and $g_\ell$ are sampled from the single functions $x\mapsto f(x,a_{i,j})$ and $x\mapsto g(x,a_{i,j})$, respectively, so Assumption~\ref{assmp:deriv} bounds their centered second differences:
    \begin{align}\label{eq:fg-second}
      \begin{split}
        |f_{i+1}-2f_i+f_{i-1}| & \le F_2(\Delta x)^2, \\
        |g_{i+1}-2g_i+g_{i-1}| & \le G_2(\Delta x)^2.
      \end{split}
    \end{align}
    Combining \eqref{eq:fg-second} with \eqref{eq:p-bound} proves \eqref{eq:R-bound}.

    \emph{Step 4: The scalar recurrence and its iteration.}
    Combining the weight and remainder estimates gives a scalar recurrence for $K_j$.
    In \eqref{eq:z-recurrence}, the weights are nonnegative by \eqref{eq:gamma-nonneg}, so each $z_{\ell,j}$ may be replaced by $K_j$, which dominates it by the definition \eqref{eq:z-def} of $K_j$.
    Together with the bound \eqref{eq:R-bound} on $R_i$, this gives
    \begin{align}
      \begin{split}
        z_{i,j-1}
         & \le(\gamma_i^0+\gamma_i^++\gamma_i^-)K_j
        +\Delta t(2F_2K_j+G_2)                                   \\
         & \le\left[1+(3F_1+2F_2)\Delta t\right]K_j+G_2\Delta t,
      \end{split}
    \end{align}
    where the second inequality uses \eqref{eq:gamma-sum}.
    Taking the maximum over $i$ proves the recurrence \eqref{eq:K_recurrence}.

    Finally, we bound the terminal second differences and iterate the scalar recurrence backward in time.
    At $j=N_t$, periodicity of $v_T$ and its spatial $\calC^2$ regularity imply
    \begin{align}
      z_{i,N_t}
      =\frac{v_T(x_i+\Delta x)-2v_T(x_i)+v_T(x_i-\Delta x)}{(\Delta x)^2}
      \le\|\partial_{xx} v_T\|_{L^\infty(\Omega)}.
    \end{align}
    Hence $K_{N_t}\le\|\partial_{xx} v_T\|_{L^\infty(\Omega)}$.
    For any $j=0,\ldots,N_t$, there are $N_t-j$ updates from the terminal level $N_t$ to level $j$, and $(N_t-j)\Delta t=T-t_j$.
    Set $A\coloneqq1+(3F_1+2F_2)\Delta t$ and $n\coloneqq N_t-j$.
    Iterating \eqref{eq:K_recurrence} backward from $N_t$ gives
    \begin{align*}
      K_j
      &\le A^nK_{N_t}+G_2\Delta t\sum_{\ell=0}^{n-1}A^\ell \\
      &\le A^n\left(K_{N_t}+n\Delta t\,G_2\right).
    \end{align*}
    Since $n\Delta t=T-t_j$ and $A^n\le\exp((3F_1+2F_2)(T-t_j))$, we obtain
    \begin{align}\label{eq:K_uniform_bound}
      \begin{split}
        K_j
         & \le
        \exp\bigl((3F_1+2F_2)(T-t_j)\bigr)
        \left(K_{N_t}+(T-t_j)\,G_2\right) \\
         & \le
        \exp\bigl((3F_1+2F_2)T\bigr)
        \left(\|\partial_{xx} v_T\|_{L^\infty(\Omega)}+TG_2\right)
        =C_{\mathrm{sc}}.
      \end{split}
    \end{align}
    Estimate~\eqref{eq:K_uniform_bound} also applies when $3F_1+2F_2=0$ and is independent of $\Delta x$ and $\Delta t$.
    Therefore, $z_{i,j}\le K_j\le C_{\mathrm{sc}}$ for all $i$ and $j=0,\ldots,N_t$, which proves the claim.
  \end{proof}

\end{appendices}

\bibliography{references}%

\end{document}